\documentclass[10pt,reqno]{amsart}
\usepackage{amsmath,amscd,amssymb}
\usepackage{graphicx}
\usepackage{color}
\usepackage{url}
\usepackage{hyperref,doi}
\usepackage{enumerate}

\theoremstyle{plain}
   \newtheorem{theorem}{Theorem}[section]
   
   \newtheorem{lemma}[theorem]{Lemma}
   \newtheorem{corollary}[theorem]{Corollary}

\theoremstyle{definition}

\newtheorem{example}{Example}[section] 

\theoremstyle{remark}
 \newtheorem{remark}{Remark}[section]

\newcommand{\ZZ}{\mathbb{Z}}

\newcommand{\m}{{\it m}}

\newcommand\independent{\protect\mathpalette{\protect\independenT}{\perp}}
\def\independenT#1#2{\mathrel{\rlap{$#1#2$}\mkern2mu{#1#2}}}

\def\newop#1{\expandafter\def\csname #1\endcsname{\mathop{\rm
#1}\nolimits}}

\newop{nd}
\newop{pa}
\begin{document}
\title[Counting Markov Equivalence Classes for DAG models on Trees]{Counting Markov Equivalence Classes for DAG models on Trees}


\author{Adityanarayanan Radhakrishnan}
\address{         
Laboratory for Information and Decision Systems,\\
         and Institute for Data, Systems, and Society\\
         MIT\\
Cambridge, MA, USA}
\email{aradha@mit.edu}

\author{Liam Solus}
\address{KTH Royal Institute of Technology\\
	Stockholm, Sweden}
\email{solus@kth.se}

\author{Caroline Uhler}
\address{
         Laboratory for Information and Decision Systems,\\
         and Institute for Data, Systems, and Society\\
         MIT\\
         Cambridge, MA, USA}
\email{cuhler@mit.edu}

\date{28 May 2017}

\begin{abstract}
DAG models are statistical models satisfying a collection of conditional independence relations encoded by the nonedges of a directed acyclic graph (DAG) $\mathcal{G}$.  
Such models are used to model complex cause-effect systems across a variety of research fields.  
From observational data alone, a DAG model $\mathcal{G}$ is only recoverable up to \emph{Markov equivalence}.  
Combinatorially, two DAGs are Markov equivalent if and only if they have the same underlying undirected graph (i.e. skeleton) and the same set of the induced subDAGs $i\to j \leftarrow k$, known as immoralities.  
Hence it is of interest to study the number and size of Markov equivalence classes (MECs).
In a recent paper, the authors introduced a pair of generating functions that enumerate the number of MECs on a fixed skeleton by number of immoralities and by class size, and they studied the complexity of computing these functions.  
In this paper, we lay the foundation for studying these generating functions by analyzing their structure for trees and other closely related graphs.  
We describe these polynomials for some important families of graphs including paths, stars, cycles, spider graphs, caterpillars, and complete binary trees.  
In doing so, we recover important connections to independence polynomials, and extend some classical identities that hold for Fibonacci numbers.  
We also provide tight lower and upper bounds for the number and size of MECs on any tree.  
Finally, we use computational methods to show that the number and distribution of high degree nodes in a triangle-free graph dictates the number and size of MECs.  
\end{abstract}

\maketitle
\thispagestyle{empty}

\section{Introduction}
\label{sec: introduction}
A graphical model based on a directed acyclic graph (DAG), known as a \emph{DAG model} or \emph{Bayesian network}, is a type of statistical model used to model complex cause-and-effect systems.  
DAG models are popular in numerous areas of research including computational biology, epidemiology, environmental management, and sociology~\cite{A11, Friedman_2000, Pearl_2000, Robins_2000, Spirtes_2001}.
Given a DAG $\mathcal{G} := ([p],A)$ with nodes $[p]=\{1, \dots , p\}$ and arrows $i \rightarrow j\in A$, the DAG model associates to each node $i\in[p]$ of $\mathcal{G}$ a random variable $X_i$.  
The collection of non-arrows of $\mathcal{G}$ encode those conditional independence (CI) relations typical of cause-effect relationships:
$$
X_i \independent X_{\nd(i)\backslash\pa(i)} \,\mid\, X_{\pa(i)},
$$
where $\nd(i)$ and $\pa(i)$ respectively denote  the \emph{nondesendents} and \emph{parents} of the node $i$ in $\mathcal{G}$.  
A probability distribution $\mathbb{P}$ is said to satisfy the \emph{Markov assumption} with respect to $\mathcal{G}$ if it entails these CI relations, and the DAG model associated to $\mathcal{G}$ is the complete set of all such joint probability distributions.  
The global consequences of the Markov assumption in terms of CI relations can be captured via the combinatorics of the DAG $\mathcal{G}$ with a notion of directed separation called \emph{$d$-separation} \cite[Chapter 3]{DSS08}.  
Unfortunately, multiple DAGs can encode the same set of CI relations.  
Such DAGs are said to be \emph{Markov equivalent}, and the complete collection of DAGs encoding the same set of CI relations as $\mathcal{G}$ is called the \emph{Markov equivalence class} (MEC) of $\mathcal{G}$.  
Verma and Pearl show in \cite{VP92} that a MEC is combinatorially determined by the underlying undirected graph $G$ (or \emph{skeleton}) of $\mathcal{G}$ and the placement of immoralities, i.e.~induced subgraphs of the form $i\rightarrow j\leftarrow k$.

From observational data, the underlying DAG $\mathcal{G}$ of a DAG model can only be determined up to Markov equivalence.  
It is therefore of interest to gain a combinatorial understanding of MECs, in particular their number and sizes. 
The literature on the MEC enumeration problem can be summarized via the following three perspectives:  (1) count the number of MECs on all DAGs on $p$ nodes \cite{GP01}, (2) count the number of MECs of a given size \cite{G06,S03,W13}, or (3) determine the size of a specific MEC \cite{HJY15, HY16}.  
In \cite{GP01}, the authors approach perspective (1) computationally and compute the number of MECs for all DAGs on $p\leq 10$ nodes.  
In \cite{G06,S03,W13}, the authors provide partial results for perspective (2) using inclusion-exclusion formulae that work nicely for small MECs sizes.  
Then in \cite{HJY15, HY16}, the authors explore efficient techniques for computing the size of a fixed MEC via algorithms that manipulate \emph{$v$-rooted} and \emph{core} subgraphs of chordal graphs.  
Recently, \cite{RSU17} addresses this question from a new perspective by introducing a pair of generating functions that enumerate the number of MECs on a \emph{fixed skeleton} $G=(V,E)$ by number of immoralities in each class and by class size.  
Their results reveal connections to graphical enumeration problems that are well-studied from the perspective of combinatorial optimization.  
A main goal of this paper is make explicit these connections and use them to study the generating functions of \cite{RSU17} for sparse graphs.  


Throughout, we use curly letters for DAGs, such as $\mathcal{G}$, and script letters for the corresponding undirected graph (i.e.~skeleton), such as $G$. In addition, we use $A$ to denote a collection of arrows and $E$ to denote a collection of undirected edges. 
The first generating function is the graph polynomial
$$
M(G;x) := \sum_{k\geq0}\m_k(G)x^k, 
$$
where $\m_k(G)$ denotes the number of MECs with skeleton $G$ that contain precisely $k$ immoralities.  
The degree of $M(G;x)$, denoted $\m(G)$, is called the \emph{immorality number} of $G$, and it counts the maximum number of immoralities possible in an MEC with skeleton $G$.  
The second generating function is the arithmetic function
$$
S(G;x) := \sum_{k\geq0}\frac{s_k(G)}{k^x}, 
$$
where $s_k(G)$ denotes the number of MECs with skeleton $G$ that have size $k$.  
We let $M(G):=M(G;1)=S(G;0)$ denote the total number of MECs with skeleton $G$.  
In \cite{RSU17}, the authors showed that computing a DAG with $\m(G)$ immoralities is an NP-hard problem, and that $S(G;x)$ is a complete graph isomorphism invariant for all connected graphs on $p\leq 10$ nodes.  
Otherwise, very little is known about the structure of these generating functions.  

In this paper, we lay the foundation for the study of the graph polynomial $M(G;x)$ by providing a detailed analysis of its properties for trees (and their closely related graphs).  
Within this context, we draw explicit connections between properties of $M(G;x)$ and the \emph{independence polynomial} of $G$; i.e. the graph polynomial 
$
I(G;x):=\sum_{k\geq0}\alpha_k(G)x^k,
$
where $\alpha_k(G)$ denotes the number of pairwise disjoint $k$-subsets of vertices (\emph{independent sets}) of $G$.  

The remainder of this paper is structured as follows:
In Section~\ref{sec: some first examples} we compute $M(G;x)$ and $S(G;x)$ for some fundamental examples, including paths, cycles, and stars.  
We find that $M(G;x)$ coincides with an independence polynomial for paths and cycles, therein providing connections to Fibonacci numbers and Fibonacci-like sequences.  
Paths and stars give tight bounds on the number of independence sets in a tree \cite{PT82}. 
We show in Section~\ref{sec: bounding the size and number of mecs on trees} that they also provide tight upper and lower bounds for the number and sizes of MECs on a tree. 
In Section~\ref{sec: classic families of trees} we then use $M(G;x)$ for stars and paths to compute $M(G;x)$ and $M(G)$ for families of trees that are significant in both mathematical and statistical settings. 
The graphs analyzed include spider graphs, caterpillar graphs, and complete binary trees.  
In the case of spider graphs, the resulting formulae yield generalizations of classic identities known for Fibonacci numbers, and reveal a multivariate extension of $M(G;x)$ exhibiting nice combinatorial properties that can be recursively computed for any tree.  
In Section~\ref{sec: beyond trees}, we use computational methods to examine properties of $M(G;x)$ and $M(G)$ for the more general family of triangle-free graphs.  
The results of \cite{RSU17} and those of Sections~\ref{sec: some first examples},~\ref{sec: bounding the size and number of mecs on trees}, and~\ref{sec: classic families of trees} exhibit an underlying relationship between the number and size of MECs and the number of cycles and high degree nodes in the graph.  
Using a program first described in \cite{RSU17}, we study this connection by examining data collected on MECs for all connected graphs on $p\leq 10$ nodes.  
We compare class size and the number of MECs per skeleton to skeletal features including average degree, maximum degree, clustering coefficient, and the ratio of number of immoralities in the MEC to the number of induced $3$-paths in the skeleton. 
Unlike $S(G;x)$, the polynomial $M(G;x)$ is not a complete graph isomorphism invariant over all connected graphs on $p\leq10$ nodes.  
However, using this program, we observe that it is such an invariant when restricted to triangle-free graphs.  
  
\section{Some First Examples}
\label{sec: some first examples}


In this section, we compute the generating functions $M(G;x)$ and $S(G;x)$ for paths, cycles, stars, and bistars.  
We show that $M(G;x)$ are independence polynomials for all paths and cycles.  
Similarly, we show that for the star graphs $M(G;x)$ has nonzero coefficients given by the binomial coefficients, which are precisely the coefficients of its corresponding independence polynomial.  
These examples are fundamental to the theory developed in Sections~\ref{sec: bounding the size and number of mecs on trees} and~\ref{sec: classic families of trees}, in which we bound the number and size of MECs on trees and compute $M(G;x)$ for more general families of graphs using paths and stars.  

Recall that the \emph{$p$-path} is the (undirected) graph $I_p := ([p],E)$ for which $E := \{\{i,i+1\} : i\in[n-1]\}$, and the \emph{$p$-cycle} is the (undirected) graph $C_p := ([p],E)$ for which $E := \{\{i,i+1\} : i\in[n-1]\}\cup\{\{1,n\}\}$.  
We also define the graph $G_p(q_1,q_2,\ldots,q_p)$ to be the undirected graph given by attaching $q_i$ leaves to node $i$ of the $p$-path $I_p$.  
The \emph{$p$-star} is the graph $G_1(p)$ and the \emph{$p,q$-bistar} is the graph $G_2(p,q)$.  
The \emph{center} node of $G_1(p)$ is its unique node of degree $p$.  

\subsection{Paths and cycles}
\label{subsec: paths and cycles}
We introduce two well-studied combinatorial sequences, and their associated polynomial filtrations that will play a fundamental role in the formulae computed in this section as well as in Sections~\ref{sec: bounding the size and number of mecs on trees} and~\ref{sec: classic families of trees}.  
Recall that the \emph{$p^{th}$ Fibonacci number} $F_p$ is defined by the recursion
$$
F_0 := 1\quad F_1 := 1, \quad \mbox{and} \quad F_p := F_{p-1}+F_{p-2} \quad\mbox{ for $p\geq2$.}
$$
The \emph{$p^{th}$ Fibonacci polynomial} is defined by
$$
F_p(x) := \sum_{k=0}^{\lfloor\frac{p}{2}\rfloor}{p-k\choose k}x^k,
$$
and it has the properties that $F_p(1) = F_p$ for all $p\geq 1$ and $F_p(x) = F_{p-1}(x)+xF_{p-2}(x)$ for all $p\geq2$.  
Analogously, the $p^{th}$ Lucas number $L_p$ is given by the Fibonacci-like recursion
$$
L_0 := 2\quad L_1 := 1, \quad \mbox{and} \quad L_p := L_{p-1}+L_{p-2} \quad\mbox{ for $p\geq2$.}
$$
The $p^{th}$ \emph{Lucas polynomial} is given by 
$$
L_0(x) := 2\quad L_1(x) := 1, \quad \mbox{and} \quad L_p(x) := L_{p-1}(x)+xL_{p-2}(x) \quad\mbox{ for $p\geq2$.}
$$

It is a well-known that the independence polynomial of the $p$-path is equal to the  $(p+1)^{st}$ Fibonacci polynomial and the independence polynomial of the $p$-cycle is given by the $p^{th}$ Lucas polynomial; i.e.
$$
I(I_p;x) = F_p(x) \quad \mbox{and} \quad I(C_p;x) = L_p(x).
$$
With these facts in hand we prove the following theorem.  
\begin{theorem}
\label{thm: path and cycle polynomials}
For the path $I_p$ and the cycle $C_p$ on $p$ nodes we have that
$$
M(I_p; x) = F_{p-1}(x) \quad \mbox{ and} \quad M(C_p:x) = L_p(x)-1.
$$
In particular, the number of MECs on $I_p$ and $C_p$, respectively, is 
$$
M(I_p) = F_{p-1} \quad \mbox{and} \quad M(C_p) = L_p -1,
$$  
and the maximum number of immoralities is
$$m(I_{p+2}) = m(C_p) = \left\lfloor\frac{p}{2}\right\rfloor.$$
\end{theorem}

\begin{proof}
The result follows from a simple combinatorial bijection.  
Since paths and cycles are the graphs with the property that the degree of any vertex is at most two, then the possible locations of immoralities are exactly the degree two nodes.  
That is, the unique head node $j$ in an immorality $i\rightarrow j\leftarrow k$ must be a degree two node.  
In the path $I_p$, this corresponds to all $p-2$ non-leaf vertices, and for the cycle $C_p$ this is all the vertices of the graph.  
Notice then that no two adjacent degree two nodes can simultaneously be the unique head node of an immorality, since this would require one arrow to be bidirected.  
Thus, a viable placement of immoralities corresponds to a choice of any subset of degree two nodes that are mutually non-adjacent, i.e. that form an independent set.  

Conversely, given any independent set in $I_p$, a DAG can be constructed by placing the head node of an immorality at each element of the set and directing all other arrows in one direction.  
Similarly, this works for any nonempty independent set in $C_p$.  
(Notice that any MEC on the cycle must have at least one immorality since all DAGs have at least one sink node.)
The resulting formulas are then 
$$
M(I_p; x) = I(I_{p-2}; x) = F_{p-1}(x) \quad \mbox{and} \quad M(C_p:x) = I(C_p; x) -1 = L_p(x)-1,
$$
which completes the proof.
\end{proof}

We now compute the generating functions $S(I_p;x)$ and $S(C_p;x)$.  
The desired formulae follow naturally from the description of the placement of immoralities given in Theorem~\ref{thm: path and cycle polynomials}.

\begin{theorem}
\label{thm: path vectors}
The number $s_\ell(I_p)$ of MECs of size $\ell$ with skeleton $I_p$ is the number of compositions $c_1+\cdots+c_{k+1} = p-k$  of $p-k$ into $k+1$ parts that satisfy
$$
\ell = \prod_{i = 1}^{k+1}c_i
$$
as $k$ varies from $0,1\ldots,\lfloor\frac{p}{2}\rfloor.$
\end{theorem}

\begin{proof}
Let $\mathcal{G}$ be a DAG with skeleton $I_p$. We denote the Markov equivalence class of $\mathcal{G}$ by $[\mathcal{G}]$. By the proof of Theorem~\ref{thm: path and cycle polynomials}, we know that the immorality placements in $[\mathcal{G}]$ correspond to the nodes in an independent $k$-subset $\mathcal{I}\subset[p]$ on the subpath $I_{p-2}$ of $I_p$ induced by the non-leaf nodes of $I_p$.  
The induced graph of the complement of $\mathcal{I}$ is a forest of $k+1$ paths.  
Since each member of $[\mathcal{G}]$ is a DAG with skeleton $I_p$ that has no immoralities on these $k+1$ paths, then each path contains a unique sink.  
Each independent $k$-subset yields a distinct forest of $k+1$ paths on $[p]\backslash\mathcal{I}$, which corresponds to a unique partition of $p-k$ into  $k+1$ parts.  
The formula for $s_\ell(I_p)$ is then given by considering all such possible placements of sinks on each path in the forests over all independent sets.  
\end{proof}

A similar argument using integer partitions allows us to compute the number of MECs of size $\ell$ on the $p$-cycle.  

\begin{theorem}
\label{thm: cycle vectors}
The number of MECs of size $\ell$ in the $p$-cycle is 
$$
s_\ell(C_p) = \sum_{k=1}^{\left\lfloor\frac{p}{2}\right\rfloor}\,\sum_{\substack{{\bf m} \,\in\, \mathbb{P}[p-2k+1,k,p-k], \\ \ell=\prod_{i=1}^k i^{m_i}}}\frac{p}{k}{k\choose m_1,\ldots,m_{p-2k+1}},
$$
where $\mathbb{P}[j,k,n]$ denotes the partitions of $n$ with $k$ parts with largest part at most $j$.
\end{theorem}

\begin{proof}
Since $C_p$ is a graph in which every node is degree 2, then each MEC of $C_p$ containing $k$ immoralities corresponds to an independent $k$-subset of $[p]$, and the subgraph of $C_p$ given by deleting this $k$-subset consists of $k$ disjoint paths.  
The size of this MEC is then the product of the lengths of these paths.  
So we need only count the number of such subgraphs for which this product equals $\ell$.  

To count these objects, consider that each subgraph of $C_p$ given by deleting an independent $k$-subset of $C_p$ forms a partition of the $p-k$ remaining vertices into $k$ parts with maximum possible part size being $p-2k+1$.  
Such a partition is represented by
$$
\left\langle 1^{m_1},2^{m_2},\ldots,(p-2k+1)^{m_{p-2k+1}}\right\rangle\in \mathbb{P}[p-2k+1,k,p-k],
$$
where $m_1,\ldots,m_{p-2k+1}\geq 0$ and $\sum_i m_i = k$.  
Each such partition corresponds to an unlabeled forest consisting of $m_i$ $i$-paths, and the number of subgraphs of $C_p$ isomorphic to this forest is 
$$
\frac{p}{k}{k\choose m_1,\ldots,m_{p-2k+1}}.
$$
The claim follows since the size of each corresponding MEC is $\prod_{i=1}^k i^{m_i}$.
\end{proof}

\begin{remark}
\label{rmk: lucas triangle}
It is a well-known result that the coefficient of $x^k$ in the $(p-1)^{st}$ Fibonacci polynomial is the binomial coefficient ${p-k-1\choose k}$, and that this is also the number of compositions of $p-k$ into $k+1$ parts.  
The former result says that the $(p-1)^{st}$ Fibonacci polynomial has coefficients given by the $(p-1)^{st}$ diagonal of Pascal's triangle, and so the latter result gives a compositional interpretation of the corresponding entry in Pascal's Triangle; see Figure~\ref{fig: pascal-and-lucas-triangles} (left). 
In this section, we saw that this compositional interpretation of ${p-k-1\choose k}$ results in the proof of Theorem~\ref{thm: path vectors}.  

	\begin{figure}[!t]
	\centering
	\includegraphics[width=0.75\textwidth]{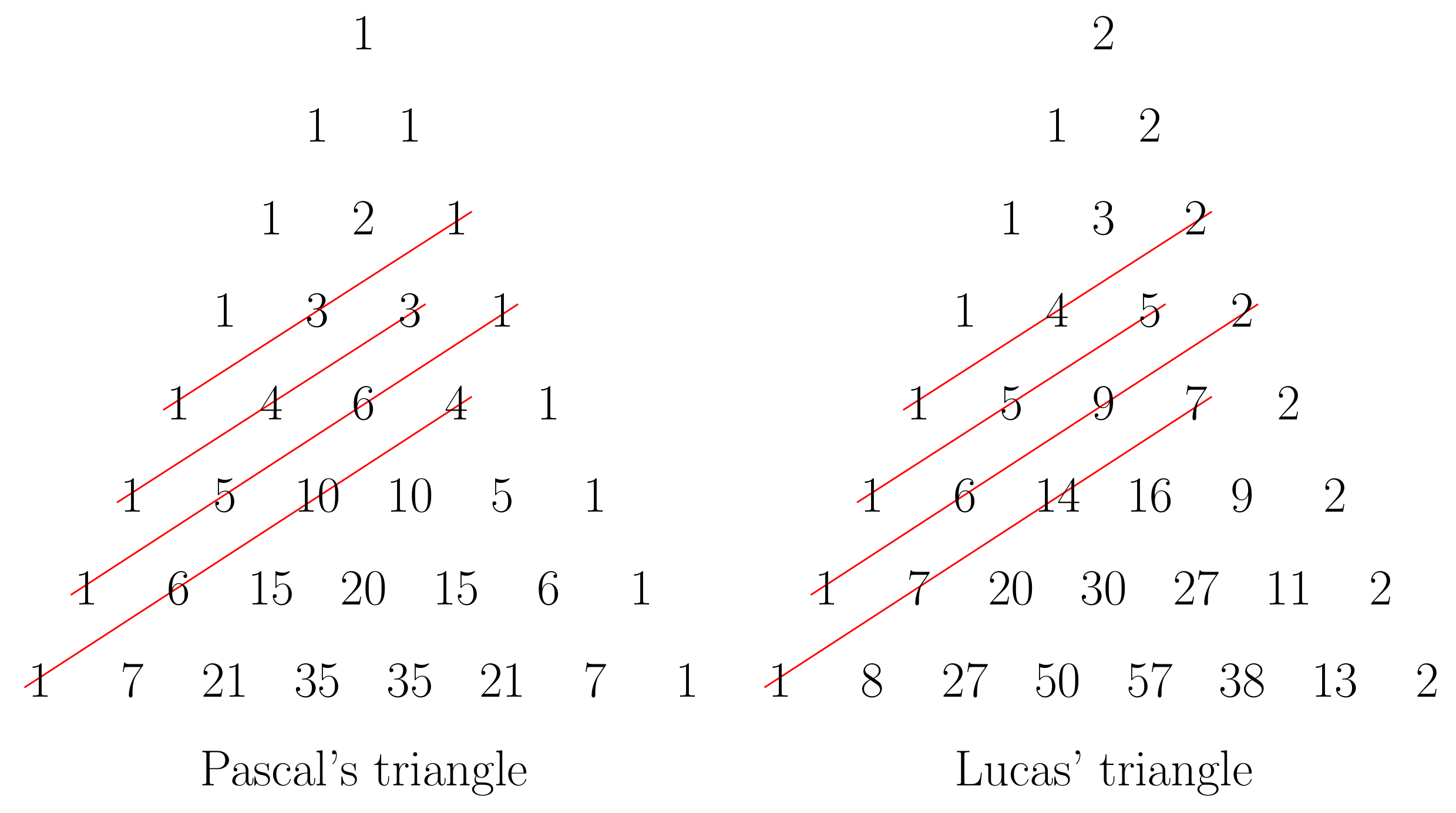}
	\caption{
	The $p^{th}$ diagonal of each triangle is the coefficient vector of $F_p(x)$ and $L_p(x)$, respectively.  }
	\label{fig: pascal-and-lucas-triangles}
	\end{figure}
Analogously, the $p^{th}$ diagonal of a second triangle, called \emph{Lucas' triangle} in \cite{BS15}, corresponds to the coefficients of the $p^{th}$ Lucas polynomial.  
This triangle is depicted on the right in Figure~\ref{fig: pascal-and-lucas-triangles}.
Thus, the proof of Theorem~\ref{thm: cycle vectors} results in a combinatorial interpretation of the entries of this triangle via partitions.  
In particular, the entry of the Lucas triangle corresponding to the $k^{th}$ coefficient of $L_p(x)$ is 
$$
[x^k].L_p(x) = \sum_{{\bf m}\in \mathbb{P}[p-2k+1,k,p-k]}\frac{p}{k}{k\choose m_1,\ldots,m_{p-2k+1}}.
$$
Moreover, the binomial recursion on the triangle implies that these coefficients satisfy the identity
$$
[x^k].L_p(x) = [x^{k-1}].L_{p-2}+[x^k].L_{p-1}.
$$
To the best of the authors' knowledge, such a partition identity is new to the combinatorial literature.
\end{remark}

%


\subsection{Stars and bistars}
\label{subsec: stars and bistars}
We now study the star and bistar graphs, $G_1(p)$ and $G_2(p,q)$.  
An example of a star and a bistar is given in Figure~\ref{fig: stars and bistars}.
	\begin{figure}[!b]
	\centering
	\includegraphics[width=0.5\textwidth]{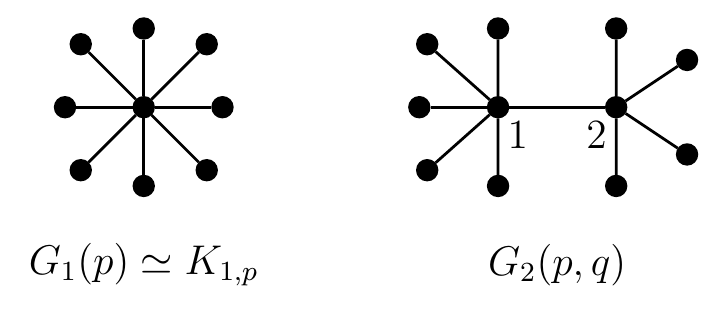}
	\caption{On the left is a star and on the right is a bistar.}
	\label{fig: stars and bistars}
	\end{figure}
The number of MECs on stars and their sizes will play an important role in Sections~\ref{sec: bounding the size and number of mecs on trees} and~\ref{sec: classic families of trees}.

\begin{theorem}
\label{thm: stars}
The MECs on the $p$-star $G_1(p)$ have the polynomial generating function
$$
M(G_1(p);x) = 1+\sum_{k\geq2}{p\choose k}x^{{k\choose 2}}.
$$
In particular, 
$$
M(G) = 2^p-p.
$$
Moreover, the corresponding class sizes are 
$$
s_1(G_1(p)) = 2^p-p+1
\qquad 
\mbox{and}
\qquad
 s_{p+1}(G_1(p)) = 1.
 $$  
\end{theorem}

\begin{proof}
Any immorality $i \rightarrow j \leftarrow k$ in a DAG on $G_1(P)$ must have the unique head node $j$ being the center node of $G_1(P)$, and the tail nodes $i$ and $k$ must be leaves of $G_1(p)$.  
It follows that each MEC on $G_1(p)$ having at least one immorality is given by selecting any $k$-subset of the $p$ leaves for $k\geq2$ to be directed towards the center node and then directing all other edges outwards.  
Each such $k$-subset yields a unique MEC of size one containing ${k\choose 2}$ immoralities.  
The final MEC is the class containing no immoralities.  
This class consists of all DAGs on $G_1(p)$ with a unique source node, and there are $p+1$ such DAGs.  
\end{proof}

The formulas in Theorem~\ref{thm: stars} allow us to obtain similar formulas for bistars.  
For convenience, we let 
$$
P_m := \sum_{k=1}^{m}{m\choose k}x^{{k+1\choose 2}}.
$$
It will also be helpful to label edges that have specified roles in certain MECs.  
The green edges (also labeled with $\square$) indicate that these edges cannot be involved in any immorality.  
The red arrows (also labeled with $\ast$) indicate a fixed immorality in the partially directed graph, and the blue arrows (also labeled with $\circ$) represent fixed arrows that are not in immoralities.  

\begin{theorem}
\label{thm: bistars}
The MECs on the bistar $G_2(p,q)$ have the polynomial generating function
$$
M(G_2(p,q);x) = M(G_1(p);x)P_q+ M(G_1(q);x)P_p + M(G_1(p);x) + M(G_1(q);x) - 1.
$$
In particular, 
$$
M(G_1(p,q)) = 2^{p+q+1}-p2^q-q2^p-1.
$$
Moreover, the corresponding class sizes are
$$
s_1(G_2(p,q)) = 2^{p+q+1}-p2^q-q2^p-2^p-2^q, 
$$
$$
s_{p+1}(G_2(p,q)) = 2^q-1, 
\quad
s_{q+1}(G_2(p,q)) = 2^p-1, 
\quad
\mbox{and} 
\quad
s_{p+q+2}(G_2(p,q)) = 1.
$$
\end{theorem}

\begin{proof}
To count the MECs on the bistar $G_2(p,q)$ we consider three separate cases defined in terms of the edge $\{1,2\}$.  
These three cases are:
\begin{enumerate}
	\item The edge $\{1,2\}$ is in an immorality with at least one of the $p$ leaves attached to node $1$. 
	\item The edge $\{1,2\}$ is in an immorality with at least one of the $q$ leaves attached to node $2$.
	\item The edge $\{1,2\}$ is not in an immorality. 
\end{enumerate}
The three cases are depicted in Figure~\ref{fig: bistar cases}. 
	\begin{figure}
	\centering
	\includegraphics[width=0.7\textwidth]{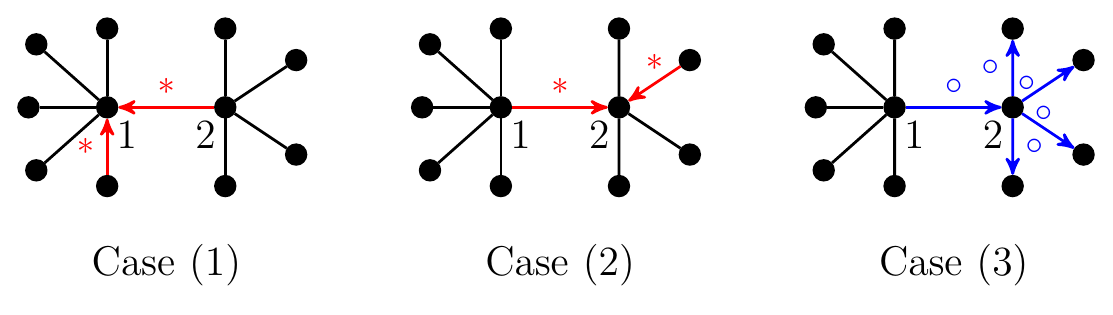}
	\caption{The three cases of the proof of Theorem~\ref{thm: bistars}.}
	\label{fig: bistar cases}
	\end{figure}
In the first case, at least one of the $p$ leaves attached to node $1$ must be in an immorality with the edge $\{1,2\}$, and the $q$ leaves attached to node $2$ can display any pattern of immoralities of the star $G_1(q)$.  
This yields $M(G_1(q);x)P_p$ MECs as counted by their number of immoralities.  
Similarly, case two yields $M(G_1(p);x)P_q$.  
In the third case, in order for the edge $\{1,2\}$ to not appear in any immorality, we need that all edges at the head of $\{1,2\}$ point towards the leaves.  
This yields $M(G_1(p);x) + M(G_1(q);x) - 1$ MECs as counted by their number of immoralities.  
Thus, 
$$
M(G_2(p,q);x) = M(G_1(p);x)P_q+ M(G_1(q);x)P_p + M(G_1(p);x) + M(G_1(q);x) - 1,
$$
and evaluating this polynomial at $1$ yields 
$$
M(G_1(p,q)) = 2^{p+q+1}-p2^q-q2^p-1.
$$

Finally, to count the classes by size we again filter by the three cases $(1),(2),$ and $(3)$.  
In the first case, there are $2^p-1$ ways for the edge $\{1,2\}$ to be in an immorality with any of the $p$ leaves at node $1$, and there are $2^q-q$ possible patterns of immoralities that can occur among the $q$ leaves at node $2$.  
One of these $2^q-q$ patterns has class size $q+1$ (the class with no immoralities), and all others have size one.  
Thus, case $(1)$ yields $2^p-1$ classes of size $q+1$ and $(2^q-q-1)(2^p-1)$ classes of size $1$.  
Similarly, case $(2)$ yields $2^q-1$ classes of size $p+1$ and $(2^p-p-1)(2^p-1)$ classes of size $1$.
In case $(3)$, if both sets of leaves contain no immoralities, then we get a single class of size $p+q+2$.  
If the $p$ leaves at node $1$ contain at least one immorality, then all leaves at node $2$ must be directed away from node $2$, yielding $2^p-p-1$ classes of size $1$.  
Similarly, if the $q$ leaves at node $2$ contain at least one immorality, then we get another $2^q-q-1$ classes of size one.  
Summing over these cases yields the desired formulae.  
\end{proof}

\section{Bounding the Size and Number of MECs on Trees}
\label{sec: bounding the size and number of mecs on trees}
We begin this section by deriving upper and lower bounds on the number of MECs for trees on $p$ nodes.  
We show that these bounds are achieved by the $(p-1)$-star $G_1(p-1)$ and the $p$-path $I_p$, respectively.  
This result parallels the classic result of \cite{PT82},  which states that the number of independent sets in a tree on $p$ nodes is bounded by the number of independent sets in $G_1(p-1)$ and $I_p$, respectively.  
\begin{theorem}
\label{thm: bound on the number of mecs for trees}
Let $T_p$ be a tree on $p$ nodes.  Then
$$
F_{p-1} = M(I_p) \leq M(T_p) \leq M(G_1(p-1)) = 2^{p-1}-p+1.
$$
\end{theorem}

\begin{proof}
We first prove the upper bound on $M(T_p)$.  
Since $T_p$ is a tree, it has precisely $p-1$ edges, and so there are $2^{p-1}$ edge orientations on $T_p$.  
Of these $2^{p-1}$ orientations, the $p$ orientations given by selecting a unique source node in $T_p$ all belong to the same MEC.  
So there are at most $2^{p-1}-p+1$ MECs for $T_p$.  
By Theorem~\ref{thm: stars}, this bound is achieved by the $(p-1)$-star $G_1(p-1)$.  

To prove the lower bound, we use a simple inductive argument.  
Notice first that the bound is true when $p\leq 5$.  
Now recall that every tree on $p$ nodes can be constructed in one of two ways: $(1)$ attaching a leaf to a degree 1 node of a tree on $p-1$ nodes, or $(2)$ attaching a leaf to a node of $T_{p-1}$ that is a neighbor of a leaf. 
Thus, given a tree $T_{p-1}$ on $p-1$ nodes, it suffices to show that when we construct $T_p$ from $T_{p-1}$ via $(1)$ or $(2)$, the number of MECs increases by at least $F_{p-3}$.  

In case $(1)$, we attach a leaf node $v$ to a leaf $u$ of $T_{p-1}$, whose only neighbor in $T_{p-1}$ is some node $w$.  
The MECs on $T_{p}$ then come in two types: either the edge $\{v,u\}$ is not in an immorality or it is in the immorality $v\rightarrow u\leftarrow w$.  
The number of classes in the first case is $M(T_{p-1})$ and the number of classes in the second case is $M(T_{p-1}\backslash u)$.  
So by the inductive hypothesis we have that
$$
M(T_p)\geq M(T_{p-1})+M(T_{p-1}\backslash u)\geq F_{p-2}+F_{p-3} = F_{p-1}.
$$
In case $(2)$, the leaf node $v$ is attached to some node $u$ of $T_{p-1}$ that has at least one leaf $w$ in $T_{p-1}$. The MECs on $T_p$ contain two disjoint types of classes: classes in which the edge $\{v,u\}$ is not in an immorality and classes containing the immorality $v\rightarrow u\leftarrow w$. Similar to the previous case, it then follows from the inductive hypothesis that
$$
M(T_p)\geq M(T_{p-1})+M(T_{p-1}\backslash w)\geq F_{p-2}+F_{p-3} = F_{p-1},
$$
which completes the proof.
\end{proof}

We now derive bounds on the size of the MEC for a fixed DAG $\mathcal{T}_p$ on the underlying undirected graph  $T_p$.   
These bounds will be computed in terms of the structure of the \emph{essential graph} $\widehat{\mathcal{T}_p}$ of the MEC $[\mathcal{T}_p]$.  
Recall that the essential graph of an MEC $[\mathcal{G}]$ is a partially directed graph $\widehat{\mathcal{G}}:=([p],E,A)$, where the collection of arrows $A$ in $\widehat{\mathcal{G}}$ are the arrows that point in the same direction for every member of the class, and the undirected edges $E$ represent the arrows that change orientation to distinguish between members of the class; see~\cite{AMP97}.  
The \emph{chain components} of $\widehat{\mathcal{G}}$ are its undirected connected components, and its \emph{essential components} are its directed connected components.  

To see why it is reasonable to work with the essential graph to derive such bounds, consider the analysis of the MEC sizes for stars and bistars given in Theorems~\ref{thm: stars} and \ref{thm: bistars}.  
In order to derive the possible sizes of these MECs, we implicitly counted all possible orientations of the undirected edges in the essential graph of each class.  
Since understanding the possible orientations of these edges is equivalent to knowing the size of the class, we will bound the size of the MEC of $\mathcal{T}_p$ in terms of the number and size of the chain components of $\widehat{\mathcal{T}_p}$.  
We will see that the computed bounds are tight, and that stars play an important role in achieving these bounds. 
We refer the reader to \cite{AMP97} for the basics relating to essential graphs.  

In the following, we assume that the essential graph $\widehat{\mathcal{T}_p}$ has chain components $\tau_1,\tau_2,\ldots,\tau_\ell$ for $\ell>0$. 
We also assume that each $\tau_i$ is \emph{nontrivial}; i.e. it has at least two vertices.
We let $\mathcal{G}(\widehat{\mathcal{T}_p})$ denote the directed subforest of the essential graph $\widehat{\mathcal{T}_p}$ consisting of all directed edges of $\widehat{\mathcal{T}_p}$, and we let $\varepsilon_1,\varepsilon_2,\ldots,\varepsilon_m$ denote its connected components.  
\begin{lemma}
\label{lem: size of MEC for trees}
Let $\mathcal{T}_p$ be a directed tree on $p$ nodes and $\widehat{\mathcal{T}_p}$ the corresponding essential graph. If $\widehat{\mathcal{T}_p}$ has chain components $\tau_1,\tau_2,\ldots,\tau_\ell$, then the size of the Markov equivalence class $[\mathcal{T}_p]$ is
$$
\#[\mathcal{T}_p] = \prod_{i=1}^\ell|V(\tau_i)|.
$$
\end{lemma}
\begin{proof}
Each element of $[\mathcal{T}_p]$ corresponds to one of the ways to direct the components $\tau_1,\ldots,\tau_\ell$, each of which is a tree.  
Suppose we directed $\tau_i$ so that it has two source nodes $s_1$ and $s_2$.  
Then along the unique path between $s_1$ and $s_2$ in the directed $\tau_i$, there must lie an immorality that is not present in $\widehat{\mathcal{T}_p}$.  
Thus, the only admissible directions of the components $\tau_i$ have no more than one source node.  
Since every DAG has at least one source node, the number of admissible directions of each $\tau_i$ is precisely the number of ways to pick the unique source node of $\tau_i$.  
This is precisely the number of vertices in $\tau_i$, thereby completing the proof.  
\end{proof}

\begin{theorem}
\label{thm: bounding the size of an MEC for trees}
Let $\mathcal{T}_p$ be a directed tree on $p$ nodes and $\widehat{\mathcal{T}_p}$ the corresponding essential graph.  
Suppose that $\widehat{\mathcal{T}_p}$ has $\ell>0$ chain components $\tau_1,\tau_2,\ldots,\tau_\ell$ and that the directed subforest $\mathcal{G}(\widehat{\mathcal{T}_p})$ of $\widehat{\mathcal{T}_p}$ has $m\geq 0$ connected components $\varepsilon_1,\varepsilon_2,\ldots,\varepsilon_m$.  
Then 
$$
2^\ell\leq \#[\mathcal{T}_p] \leq \left(\frac{p-m}{\ell}\right)^\ell.
$$
\end{theorem}
\begin{proof}
Notice first that the lower bound is immediate from Lemma~\ref{lem: size of MEC for trees} and the assumption that each $\tau_i$ is nontrivial.
So it only remains to verify the proposed upper bound. 

Let $\ell_i$ denote the number of chain components that are adjacent to $\varepsilon_i$ for all $i\in[m]$.  
Since the chain components $\tau_1,\ldots,\tau_\ell$ are all disjoint, it follows that 
$$
1+\ell_i\leq |V(\varepsilon_i)|
$$
for all $i\in[m]$.  
Therefore, a lower bound on the size of the number of nodes in the directed subforest $\mathcal{G}(\widehat{\mathcal{T}_p})$ is given by
$$
m+\sum_{i=1}^m \ell_i
\leq
|V(\mathcal{G}(\widehat{\mathcal{T}_p}))|.
$$
A closed form for the sum $\sum_{i=1}^m \ell_i$ is recovered as follows.  
Consider a complete bipartite graph $K_{\ell,m}$ whose vertices are partitioned into two blocks $A$ and $B$ where $|A|=\ell$ and $|B|=m$.  
The possible ways to assemble the components $\tau_1,\ldots,\tau_\ell$ and $\varepsilon_1,\ldots,\varepsilon_m$ into an essential tree are in bijection with the spanning trees of $K_{\ell,m}$.  
For any such spanning tree $T$ of $K_{\ell,m}$, each edge of $T$ has exactly one vertex in each of $A$ and $B$.  
Thus, 
$$
\sum_{i=1}^m \ell_i 
= \sum_{v\in A}\deg_{T}(v)
= \sum_{v\in B}\deg_{T}(v).  
$$
Since $T$ is a tree, it follows that
\begin{equation}
\label{eq_1}
\sum_{i=1}^m \ell_i = \frac{\sum_{v\in A}\deg_{T}(v)+\sum_{v\in B}\deg_{T}(v)}{2} = \ell + m-1.
\end{equation}
Therefore, 
$$
2m+\ell-1
\leq
|V(\mathcal{G}(\widehat{\mathcal{T}_p}))|. 
$$
Moreover, since $\mathcal{T}_p$ has $p$ vertices, and each edge of a spanning tree of $K_{\ell,m}$ corresponds to exactly one of the vertices shared by $\mathcal{G}(\widehat{\mathcal{T}_p})$ and the chain components $\tau_1,\ldots, \tau_\ell$, then we have that
\begin{equation}
\label{eqn: chain components vertex sum}
\sum_{j=1}^\ell |V(\tau_j)| = p+m+\ell-1-|V(\mathcal{G}(\widehat{\mathcal{T}_p}))|.
\end{equation}
Now by Lemma~\ref{lem: size of MEC for trees} and the arithmetic-geometric mean inequality, we have
\begin{equation*}
\begin{split}
\#[\mathcal{T}_p]=\prod_{j=1}^\ell |V(\tau_j)|\leq\left(\frac{\sum_{j=1}^\ell |V(\tau_j)|}{\ell}\right)^\ell.
\end{split}
\end{equation*}

Thus, by applying equation~\ref{eqn: chain components vertex sum}, we conclude that
\begin{equation*}
\begin{split}
\#[\mathcal{T}_p]&\leq\left(\frac{p+m+\ell-1-|V(\mathcal{G}(\widehat{\mathcal{T}_p}))|}{\ell}\right)^\ell\leq\left(\frac{p+m+\ell-1-(2m+\ell-1)}{\ell}\right)^\ell,
\end{split}
\end{equation*}
and so $\#[\mathcal{T}_p]\leq ((p-m)/\ell)^\ell$, which completes the proof.
\end{proof}

We now examine the tightness of the bounds in Theorem~\ref{thm: bounding the size of an MEC for trees} by considering some special cases. 
Notice first that the lower bound is tight exactly when each chain component is a single edge.  
The upper bound is tight exactly when $|V(\mathcal{G}(\widehat{\mathcal{T}_p}))|=2m+\ell-1$ and each chain component has exactly $\frac{p-m}{\ell}$ vertices.  
\begin{corollary}
\label{cor: tightness when m=1}
Suppose $\mathcal{G}(\widehat{\mathcal{T}_p})$ has precisely one connected component, i.e., $\mathcal{G}(\widehat{\mathcal{T}_p})$ is a directed tree.  
Then 
$$
2^\ell\leq\#[\mathcal{T}_p]\leq\left(\frac{p-1}{\ell}\right)^\ell, 
$$
and every directed tree $\mathcal{T}_p$ for which the upper bound is tight has the same subtree $\mathcal{G}(\widehat{\mathcal{T}_p})$, namely $G_1(\ell)$ with all edges directed inwards.  
\end{corollary}
\begin{proof}
The statement of the bounds is immediate from Theorem~\ref{thm: bounding the size of an MEC for trees}. So we only need to verify the claim on the tightness of the upper bound.  
It follows from the more general bounds described above, that the upper bound is tight exactly when $|V(\mathcal{G}(\widehat{\mathcal{T}_p}))|=\ell+1$ and each chain component has exactly $\frac{p-1}{\ell}$ vertices.
Since the chain components $\tau_1,\ldots,\tau_\ell$ are all distinct and $\mathcal{G}(\widehat{\mathcal{T}_p})$ is a directed tree with $\ell+1$ vertices, then each $\tau_j$ is adjacent to exactly one of the $\ell$ vertices of $\mathcal{G}(\widehat{\mathcal{T}_p})$, and there remains only one vertex to connect these $\ell$ vertices.  
Therefore, the skeleton of $\mathcal{G}(\widehat{\mathcal{T}_p})$ is the star $G_1(k+1)$. 
Moreover, since all essential edges in $\widehat{\mathcal{T}_p}$ are exactly the edges of $\mathcal{G}(\widehat{\mathcal{T}_p})$, then all edges of $\mathcal{G}(\widehat{\mathcal{T}_p})$ must be directed inwards towards the center node.  
An example of a graph for which this upper bound is tight is presented on the left in Figure~\ref{fig: tight upper bound graphs}.
\end{proof}
	\begin{figure}
	\centering
	$\begin{array}{c c c}
	\includegraphics[width=0.33\textwidth]{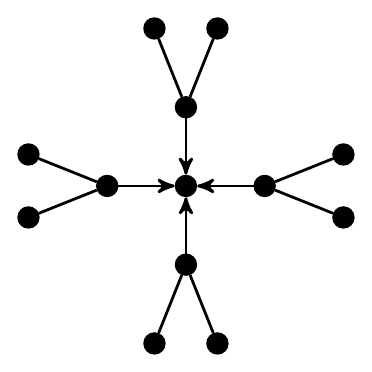} \qquad&\qquad
	\includegraphics[width=0.33\textwidth]{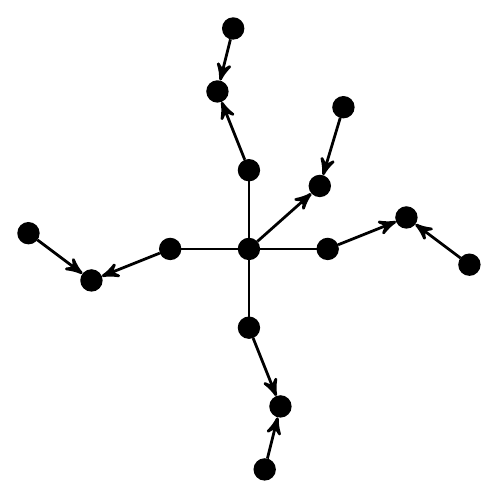} \\
	\end{array}$
	\caption{ Graphs for which the bounds in Corollary~\ref{cor: tightness when k=1} are tight when $m=1$ (left) and when $k=1$ (right).}
	\label{fig: tight upper bound graphs}
	\end{figure}
	

\begin{corollary}
\label{cor: tightness when k=1}
Suppose $\widehat{\mathcal{T}_p}$ has precisely one chain component $\tau_1$.  
Then 
$$
m\leq\#[\mathcal{T}_p]\leq p-2m, 
$$
and both bounds are tight when $\tau_1=G_1(m-1)$.  
\end{corollary}
\begin{proof}
By Lemma~\ref{lem: size of MEC for trees} we know that $\#[\mathcal{T}_p] = |V(\tau_1)|$; so the bounds presented here are bounds on the size of the vertex set of the chain component $\tau_1$.  
Since the connected components $\varepsilon_1,\ldots,\varepsilon_m$ of $\mathcal{G}(\widehat{\mathcal{T}_p})$ are all disjoint, we know that $\tau_1$ contains at least $m$ vertices.  
On the other hand, since each $\varepsilon_i$ contains at least one immorality and attaches to $\tau_1$ at precisely one node, then each $\varepsilon_i$ contains at least two nodes that are not also nodes of $\tau_i$. 
A graph for which the bounds are simultaneously tight is depicted on the right in Figure~\ref{fig: tight upper bound graphs}.  
Notice that the chain component $\tau_1$ is $G_1(m-1)$.  
\end{proof}

Corollary~\ref{cor: tightness when m=1} and Corollary~\ref{cor: tightness when k=1} suggest the important role of the maximum degree of a graph for the size of MECs. 
This is further supported and discussed via the results in the next section and the simulations in Section~\ref{subsec: skeletal structure in relation to the number and size of MECs}.

\section{Classic Families of Trees}
\label{sec: classic families of trees}
In this section, we study some classic families of trees that arise naturally in both, applied and theoretical contexts.  
Namely, we will study the graph polynomials $M(G;x)$ for \emph{spider graphs, caterpillar graphs}, and \emph{complete binary trees}.  
A \emph{spider graph} (or \emph{star-like tree}) is any tree containing precisely one node with degree greater than two, a \emph{caterpillar graph} is any tree for which deleting all leaves results in a path, and a complete binary tree is a tree for which every nonleaf node (except for possibly a root node) has precisely three neighbors. 
Caterpillars and complete binary trees play important roles for modeling events in time, as for example in phylogenetics. 
Caterpillars and spiders also provide large families of supporting examples for long-standing conjectures about well-studied generating functions associated to trees. 
Alavi, Maldi, Schwenk, and Erd\"os conjectured that the independence polynomial of every tree is unimodal \cite{AMSE87}, and Stanley conjectured that the chromatic symmetric function is a complete graph isomorphism invariant for trees \cite{S95}.  
In \cite{LM02,LM03} and \cite{MMW08} the authors, respectively, verify that these conjectures hold for caterpillars and (some) spiders.  
We show in the following that these important families of graphs also yield nice properties for the generating polynomial $M(G;x)$.  

In Section~\ref{subsec: spiders} we provide a formula for $M(G;x)$ for spider graphs that generalizes our formula for stars and paths given in Section~\ref{sec: some first examples}.  
Using these formulae we compute expressions for $M(G)$ that extend classical identities of the Fibonacci numbers.  
The methods for computing $M(G;x)$ for spiders generalizes to a multivariate formula for $M(G;x)$ for arbitrary trees with interesting combinatorial structure, which will also be described.  
In Section~\ref{subsec: caterpillars} we recursively compute $M(G;x)$ for the caterpillars.  
Using this recursive formula, we observe that these polynomials are all unimodal and estimate the expected number of immoralities in a randomly selected MEC on a caterpillar.  
Finally, in Section~\ref{subsec: complete binary trees} we compute the number of MECs for a complete binary tree, and study the rate at which this value increases.  

\subsection{Spiders}
\label{subsec: spiders}
We call the unique node of degree more than two in a spider its \emph{center} node.  
A spider $G$ on $n$ nodes with center node of degree $k$ corresponds to a partition $\lambda = (\lambda_1,\ldots,\lambda_k)$ of $n-1$ into $k$ parts.  
Following the standard notation, we assume $\lambda_1\geq\lambda_2\geq\cdots\geq\lambda_k>0$.  
Here, $\lambda_i$ denotes the number of vertices on the $i^{th}$ \emph{leg} of $G$; i.e., the $i^{th}$ maximal connected subgraph of $G$ in which every vertex has degree at most two.  
Conversely, given a partition $\lambda$ of $n-1$ into $k$ parts, we write $G_\lambda$ for the corresponding spider graph.  

In the following, we label the vertices of $G_\lambda$ such that $p_0$ denotes the center node and $p_{ij}$, for $1\leq i\leq k$ and $1\leq j\leq\lambda_i$, denotes the $j^{th}$ node from $p_0$ along the $i^{th}$ leg of $G_\lambda$.  
For a subset $S\subset[k]$, define the following polynomial:
$$
L(S;x) :=\left(\prod_{i\in S}M(I_{\lambda_i-1};x)\right)\left(\prod_{i\in[k]\backslash S}M(I_{\lambda_i};x)\right).
$$
We then have the following formula for the generating polynomial $M(G_\lambda;x)$.
\begin{theorem}
\label{thm: spider generating polynomials}
Let $G_\lambda$ denote the spider on $n$ nodes with center node of degree $k$ and partition $\lambda$ of $n-1$ into $k$ parts.  
If $\lambda$ has $\ell$ parts of size one, then 
$$
M(G_\lambda;x) = \sum_{j=0}^{k-\ell}\left(\sum_{S\in{[k-\ell]\choose j}}L(S;x)\right)x^jM(G_1(k-j);x).
$$
\end{theorem}

\begin{proof}
To arrive at this formula, simply notice that all possible placements of immoralities can be computed as follows:
First choose a subset of the $k-\ell$ nodes $\{p_{ij}:\lambda_i>1,i\in[k]\}$ at which to place immoralities.  
Call this set $S$.  
Since the nodes in $\{p_{ij}:\lambda_i>1,i\in[k]\}\backslash S$ are not immoralities then all remaining immoralities are either at the center node $p_0$, which are counted by $M(G_1(k-|S|);x)$, or they are further down the legs of the spider, which are counted by $L(S;x)$.  
\end{proof}

The general formula in Theorem~\ref{thm: spider generating polynomials} specializes to $M(G_1(p-1);x)$ when $\lambda = (1,1,\ldots,1)$ is the partition of $p-1$ into $p-1$ parts; i.e., when $G_\lambda = G_1(p-1)$.  
Similarly, for $k = 2$, it reduces to $M(I_p;x)$.  
It also yields a nice formula for the number of MECs on the spiders with $\lambda = (m,m,\ldots,m)$ a partition of $mk$ into $k$ parts.
\begin{corollary}
\label{cor: daddy long-legs}
For $k>1$ and $m\geq 1$, the spider $G_\lambda$ on $mk+1$ nodes with partition $\lambda = (m,m,\ldots,m)$ of $mk$ into $k$ parts has
$$
M(G_\lambda) = F_{m+1}^k-kF_{m-1}F_m^{k-1}.
$$
\end{corollary}

\begin{proof}
For $m=1$ we have that $G_\lambda = G_1(k)$, and the above formula reduces to $2^k - k = M(G_1(k))$.  
For $k>1$, we simplify the formula given in Theorem~\ref{thm: spider generating polynomials} to
$$
M(G_\lambda;x) = \sum_{j=0}^k{k\choose j}(xM(I_{m-1};x))^jM(I_m;x)^{k-j}M(G_1(k-j);x).
$$
Evaluating at $x =1$ yields
\begin{equation*}
\begin{split}
M(G_\lambda) 
&= \sum_{j=0}^k{k\choose j}F_{m-2}^jF_{m-1}^{k-j}(2^{k-j}-(k-j)),\\
&= \sum_{j=0}^k{k\choose j}F_{m-2}^j(2F_{m-1})^{k-j}-\sum_{j=0}^k{k\choose j}(k-j)F_{m-2}^j(F_{m-1})^{k-j},\\
&= (F_{m-2}+2F_{m-1})^{k-j}-kF_{m-1}(F_{m-2}+F_{m-1})^{k-1},\\
&= F_{m+1}^k-kF_{m-1}F_m^{k-1},\\
\end{split}
\end{equation*}
which completes the proof.
\end{proof}

\begin{remark}
\label{rem: special case of daddy long-legs formula}
In the special case of Corollary~\ref{cor: daddy long-legs} for which $k =2$ we have that $G_\lambda = I_{2m+1}$, and so $M(G_\lambda) = F_{2m}$ by Theorem~\ref{thm: path and cycle polynomials}.  
In Corollary~\ref{cor: daddy long-legs}, we see that the formula for $M(G_\lambda)$ given by Theorem~\ref{thm: spider generating polynomials} is computing the Fibonacci number $F_{2m}$ via a classic identity discovered by Lucas in 1876 (see for instance \cite{K11}):
$$
F_{2m} = F_{m+1}^2-2F_{m-1}F_m = F_m^2+F_{m-1}^2.
$$
Notice that the same expression does not hold for the generating polynomials:
$$
M(G_\lambda;x) \neq M(I_{m+2};x)^k - kM(I_m;x)M(I_{m+1};x)^{k-1}.
$$
This is because $M(G_1(p);x) = 1+\sum_{k\geq2}{p\choose k}x^{k\choose 2}$ as opposed to $1+\sum_{k\geq2}{p\choose k}x^k$.  
However, when the formula for $M(G_\lambda;x)$ used in the proof of Corollary~\ref{cor: daddy long-legs} is evaluated at $x = 1$, the exponents in the formula for $M(G_1(p);x)$ become irrelevant.  
For instance, in the case when $\lambda = (2,2)$, we have that
\begin{equation*}
\begin{split}
M(G_\lambda;x) &= x^2+3x+1,	\mbox{ but}	\\
M(I_{m+2};x)^k - kM(I_m;x)&M(I_{m+1};x)^{k-1} = 4x^2+2x-1.\\
\end{split}
\end{equation*}
However, evaluating both polynomials at $x=1$ results in the Fibonacci number $F_4=5$, as predicted by Corollary~\ref{cor: daddy long-legs}.  
\end{remark}

We end this section with a remark and example illustrating the more general consequences of the techniques used in the computation of $M(G_\lambda;x)$ in Theorem~\ref{thm: spider generating polynomials}.  
\begin{remark}
\label{rem: a general formula for all trees}
It is natural to ask if the recursive approach used to prove Theorem~\ref{thm: spider generating polynomials} generalizes to arbitrary trees.  
In particular, it would be nice if for any tree $T$, the polynomial $M(T;x)$ can be expressed as
\begin{equation}
\label{eqn: desired form}
M(T;x) = \sum_{\alpha=(\alpha_2,\ldots,\alpha_{n-1})\in\ZZ^{n-2}_{\geq0}}c_\alpha {\bf s}^{\alpha},
\end{equation}
where $s_i :=M(G_1(i);x)$ for $i = 2,\ldots,n-1$, ${\bf s}^\alpha:=s_2^{\alpha_2}s_3^{\alpha_3}\cdots s_{n-1}^{\alpha_{n-1}}$, and the $c_\alpha$ are polynomials in $x$ with nonnegative integer coefficients.  
On the one hand, there exists an (albeit cumbersome) recursion for computing $M(T;x)$ that generalizes the one used in Theorem~\ref{thm: spider generating polynomials}.  
On the other hand, this recursion will not yield an expression of the form in equation~(\ref{eqn: desired form}) unless it has at most one node with degree more than two.  
Instead, if we take 
$$
a_p := \sum_{k\geq2}{p-1\choose k}x^{k\choose2} 
\quad
\mbox{ and}
\quad
b_p := \sum_{k\geq2}{p-1\choose k-1}x^{k\choose2},
$$ 
then we can express $M(T;x)$ as 
\begin{equation}
\label{eqn: obtainable form}
M(T;x) = \sum_{\alpha = (\alpha_1,\alpha_2,\alpha_3)\in\ZZ^{n-2}_{\geq0}\times\ZZ^{n-2}_{\geq0}\times\ZZ^{n-2}_{\geq0}}c_\alpha {\bf s}^{\alpha_1}{\bf a}^{\alpha_2}{\bf b}^{\alpha_3},
\end{equation}
where ${\bf a}^{\alpha}$ and ${\bf b}^{\alpha}$ are defined analogously to ${\bf s}^{\alpha}$, and the $c_\alpha$ are polynomials in $x$ with nonnegative integer coefficients. 
The algorithm resulting in the expression for $M(T;x)$ given in equation~(\ref{eqn: obtainable form}) is the intuitive generalization of Theorem~\ref{thm: spider generating polynomials}. Since it is technical to formalize, we here only illustrate it with Example~\ref{ex: illustrative example}.  
\end{remark}

\begin{example}
\label{ex: illustrative example}
Consider the tree $T$ on $12$ nodes depicted in Figure~\ref{fig: illustrative example}.  
We follow the same approach for counting MECs in $T$ that we used to count the MECs in $G_\lambda$ in Theorem~\ref{thm: spider generating polynomials}. 
That is, we select a center node, choose a collection of immoralities at its nonleaf neighbors, and count the possible classes containing these immoralities.  
Thinking of node $0$ as the analogous vertex to the center node of a spider, we notice that it has precisely one nonleaf neighbor, namely node $1$.  
The MECs on $T$ with node $1$ in an immorality are counted by $xs_5^2$.  
Now consider those MECs on $T$ for which $1$ is not in an immorality.  
Analogous to the proof of Theorem~\ref{thm: spider generating polynomials}, we must consider the MECs on the $6$-star with center node $0$ and leaves $1,8,9,10,11,$ and $12$.  
Notice $b_6$ enumerates the MECs on this $6$-star that use the arrow $0\leftarrow 1$, and $a_6+1$ enumerates those MECs not using this arrow.  
For those enumerated by $b_6$, we then count the number of MECs on the induced subtree $T^\prime$ with vertex set $[7]$.  
This gives $b_6M(T^\prime;x)$. 

For the MECs enumerated by $a_6+1$, we must consider more carefully the structure of immoralities on $T^\prime$.  
The constant $1$ counts the choice of no immoralities on the $6$-star, and this yields $1M(T^\prime;x)$ MECs on $T$. 
On the other hand, $a_6$ counts those classes on the $6$-star with at least one immorality using the arrow $0\leftarrow1$.  
For these, we take node $2$ as the center node of $T^\prime$, which has precisely one non-leaf neighbor, node $3$.  
The ways in which node $3$ can be in an immorality are counted by $b_5$.  
If node $3$ is not in an immorality, then either $2$ is in an immorality or there are no immoralities on $T^\prime$.  
This yields 
$
a_6(1+b_5+xs_4).
$
Using the same techniques, we compute that $M(T^\prime;x) = s_2s_4+b_5$.  
Combining these formulae yields
\begin{equation*}
\begin{split}M(T;x) 
&=xs_5^2+s_2s_4+s_2s_4b_6+b_5+b_5b_6+xs_4a_6+a_6+a_6b_5.\\
\end{split}
\end{equation*}
In general, this iterative process of picking a center node for a tree $T$, choosing immorality placements for its nonleaf neighbors, and then enumerating the resulting possible MECs based on these choices results in an expression of the form given by equation~(\ref{eqn: obtainable form}).  
The monomial ${\bf s}^{\alpha_1}{\bf a}^{\alpha_2}{\bf b}^{\alpha_3}$ enumerates the possible placements of immoralities at the chosen sequence of center nodes and the coefficient polynomial $c_\alpha$ is enumerating the ways to fix immoralities at their nonleaf neighbors to allow for these placements. \qed
\end{example}

	\begin{figure}[!t]
	\centering
	\includegraphics[width=0.4\textwidth]{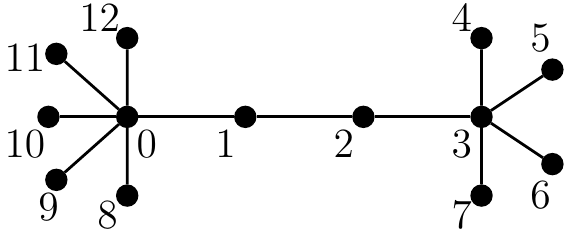}
	\caption{The tree for Example~\ref{ex: illustrative example}.}
	\label{fig: illustrative example}
	\end{figure}

Theorem~\ref{thm: spider generating polynomials} demonstrates that for some trees the expression for $M(T;x)$ given by the algorithmic approach described in Example~\ref{ex: illustrative example} can have nice coefficient polynomials $c_\alpha$.  
It is important to notice that the expression of $M(T;x)$ given in equation~(\ref{eqn: obtainable form}) is dependent of the initial choice of center node.  
However, as exhibited by Theorem~\ref{thm: spider generating polynomials}, a well-chosen initial center node and number of iterations of this decomposition can yield nice combinatorial expressions for $M(T;x)$ of the form~(\ref{eqn: desired form}) and/or~(\ref{eqn: obtainable form}).    
For example, if $T$ is the spider graph, one iteration of this decomposition initialized at the spider's center yields coefficient polynomials $c_\alpha$ that are products of Fibonacci polynomials, and when all legs are the same length, they are therefore real-rooted, log-concave, and unimodal.  
It would be interesting to know whether other families of trees yield coefficient polynomials $c_\alpha$ with nice combinatorial properties.  
Moreover, it is unclear if for every tree $T$ the polynomial $M(T;x)$ admits an expression as in equation~(\ref{eqn: desired form}). 

\subsection{Caterpillars}
\label{subsec: caterpillars}
We denote the caterpillar graph $W_p$ as
$$
W_p := 
\begin{cases}
G_{\frac{p}{2}}\left(1,1,\ldots,1\right) 		&	\mbox{if $p$ is even,}	\\
G_{\frac{p+1}{2}}\left(1,1,\ldots,1,0 \right) 	&	\mbox{if $p$ is odd.}	\\
\end{cases}
$$
	\begin{figure}[t!]
	\centering
	\includegraphics[width=0.9\textwidth]{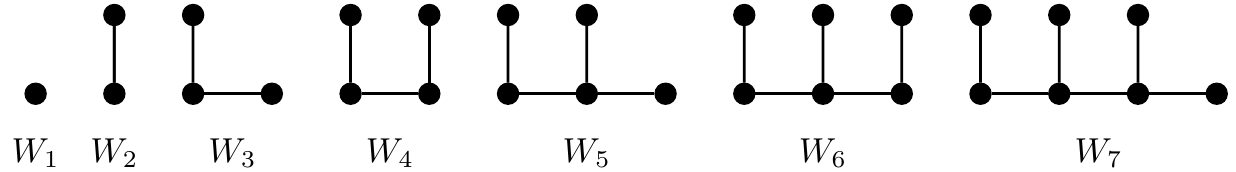}
	\caption{The first few caterpillar graphs.}
	\label{fig: caterpillars}
	\end{figure}
The first few caterpillar graphs are depicted in Figure~\ref{fig: caterpillars}.  
Since the caterpillar graphs are closely related to paths, we would expect that a similar recursive approach also works for counting the number of MECs on $W_p$.  
Indeed, with the following theorem, we provide a recursive formula for $M(W_p;x)$.  
\begin{theorem}
\label{thm: caterpillars polynomials}
Let $\mathbb{W}_p:=M(W_p;x)$ for $p\geq 1$.  
These generating polynomials satisfy the recursion with initial conditions
\begin{equation*}
\begin{split}
\mathbb{W}_1 = 1, 
\quad
\mathbb{W}_2 = 1&, 
\quad
\mathbb{W}_3 = 1+x, 
\quad
\mathbb{W}_4 = 1+2x, \\
\end{split}
\end{equation*}
and for $p\geq5$
$$
\mathbb{W}_p = \begin{cases}
	\mathbb{W}_{p-1} + x\mathbb{W}_{p-2} 	&	\mbox{for $p$ odd,}	\\
	(x+2)\mathbb{W}_{p-2} + (x^3-x^2+x-2)\mathbb{W}_{p-3} +(x^2+1)\mathbb{W}_{p-4}		&	\mbox{for $p$ even.}\\
	\end{cases}
$$
\end{theorem}

	\begin{figure}
	\centering
	\includegraphics[width=0.9\textwidth]{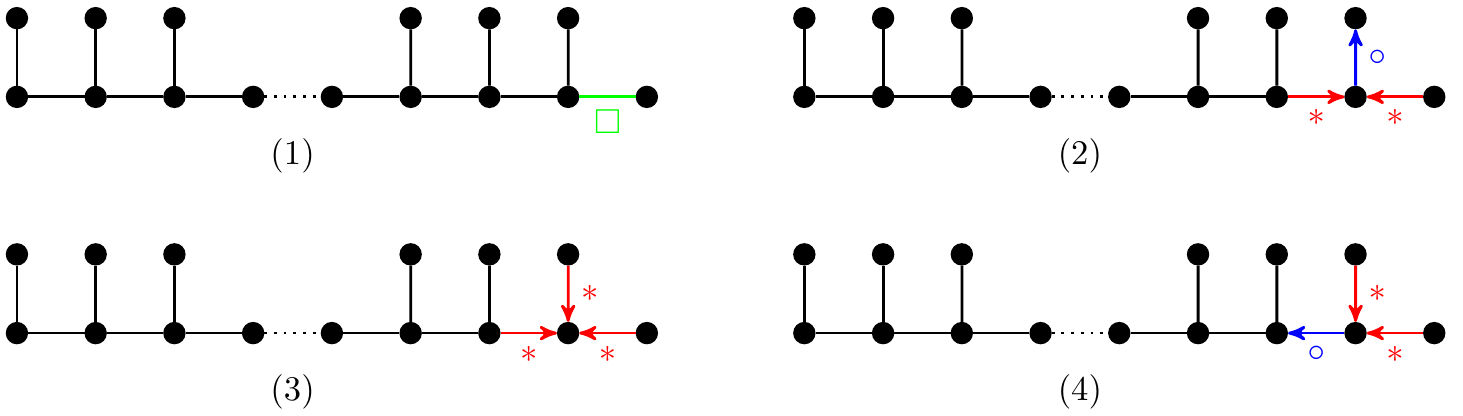}
	\caption{The four cases for the recursion on the caterpillar graph for $p$ odd. }
	\label{fig: caterpillar recursion cases}
	\end{figure}

\begin{proof}
Notice first that when $p$ is even, we can simply apply the Fibonacci recursion
$$
M(G_{\frac{p}{2}}(1,1,\ldots,1);x) = M(G_{\frac{p}{2}}(1,1,\ldots,0)) + xM(G_{\frac{p}{2}-1}(1,1,\ldots,1);x).
$$
The recursion is based on whether or not the final edge is contained within an immorality.  

Now let $p = 2k+1$ be odd.  
We first show that 
$$
\mathbb{W}_p = \mathbb{W}_{p-1} + (x^3+x)\mathbb{W}_{p-3} + x\mathbb{W}_{p-2} - x^2\sum_{j = 2}^{\left\lfloor\frac{p}{2}\right\rfloor}\mathbb{W}_{p-2j-1}.
$$
This recursion can be detected by considering the ways in which the final edge can or cannot be in an immorality.  
That is, either it is not in an immorality, or it is in an immorality with some nonempty subset of edges adjacent to it, as depicted in Figure~\ref{fig: caterpillar recursion cases}. Collectively, cases $(1)$, $(2)$, and $(3)$ yield 
$$
 \mathbb{W}_{p-1} + (x^3+x)\mathbb{W}_{p-3}
 $$
MECs.  
On the other hand, case $(4)$ yields $x\mathbb{W}_{p-2}$ minus some over-counted cases.  
The over-counted cases correspond to exactly when the first immorality to the right of the one depicted in case $(4)$ points towards the right, as depicted in Figure~\ref{fig: caterpillar overcounted cases}.  
	\begin{figure}[b!]
	\centering
	\includegraphics[width=0.9\textwidth]{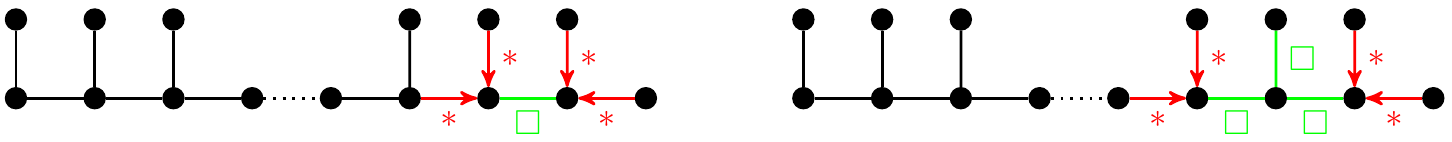}
	\caption{The over-counted cases of case $(4)$ in the caterpillar recursion for $p$ odd.}
	\label{fig: caterpillar overcounted cases}
	\end{figure}
Each such case would naturally force one more unspecified immorality.  
Thus, the total number of MECs counted by case $(4)$ is 
$$
x\mathbb{W}_{p-2} - x^2\sum_{j = 2}^{\left\lfloor\frac{p}{2}\right\rfloor}\mathbb{W}_{p-2j-1}.
$$
Since $p-1$ is even, we may apply the Fibonacci recursion to $\mathbb{W}_{p-1}$ to obtain
$$
\mathbb{W}_p - (x+1)\mathbb{W}_{p-2} = (x^3+2x)\mathbb{W}_{p-3} - x^2\sum_{j = 2}^{\left\lfloor\frac{p}{2}\right\rfloor}\mathbb{W}_{p-2j-1}.
$$
We then consider the difference between $\mathbb{W}_p - (x+1)\mathbb{W}_{p-2}$ and $\mathbb{W}_{p-2} - (x+1)\mathbb{W}_{p-4}$, and repeatedly apply the Fibonacci recursion to the even terms.  
The result is
\begin{equation*}
\begin{split}
\mathbb{W}_p - (x+2)\mathbb{W}_{p-2} &= (x^3+x-1)\mathbb{W}_{p-4} + (x^3-x^2+x-2)(\mathbb{W}_{p-3} - \mathbb{W}_{p-4}).
\end{split}
\end{equation*}
This simplifies to 
$$
\mathbb{W}_p = (x+2)\mathbb{W}_{p-2} + (x^3-x^2+x-2)\mathbb{W}_{p-3} +(x^2+1)\mathbb{W}_{p-4},
$$
thereby completing the proof.
\end{proof}
The first few polynomials $M(W_p;x)$ for $1\leq p \leq 14$, and the number of MECs on $W_p$, are displayed in Table~\ref{table: caterpillars}.
These polynomials all appear to be unimodal.  
Using the recursion in Theorem~\ref{thm: caterpillars polynomials} we can estimate that the immorality number of $W_p$ is $m(W_p)=\left\lfloor\frac{p}{2}\right\rfloor+\left\lfloor\frac{p}{4}\right\rfloor$, and that the expected number of immoralities in a randomly chosen MEC on $W_p$ approaches $\left\lfloor\frac{m(W_p)}{2}\right\rfloor$.
As an immediate corollary to Theorem~\ref{thm: caterpillars polynomials}, we get a recursion for the number of MECs $M(W_p)$.  

\begin{corollary}
\label{thm: caterpillars}
The number of MECs for the caterpillar graph $W_p$ is given by the recursion
$$
M(W_1) = 1, 
\qquad
M(W_2) = 1, 
\qquad
M(W_3) = 2, 
\qquad
M(W_4) = 3, 
$$
 and for $p\geq5$
$$
M(W_p) = 
\begin{cases}
M(W_{p-1}) + M(W_{p-2})		 		&	\mbox{if $p$ is even,}	\\
3M(W_{p-2}) + M(W_{p-4}) - M(W_{p-5}) 	&	\mbox{if $p$ is odd.}	\\
\end{cases}
$$
\end{corollary}

\begin{table}[t]
\centering
\begin{tabular}{ c | l }
 	$M(W_p)$ 	& 	$M(W_p;x)$	\\ \hline
	$1$			&	$1$	\\ \hline
	$1$			&	$1$	\\ \hline
	$2$			&	$x+1$	\\ \hline	
	$3$			&	$2x+1$	\\ \hline
	$7$			&	$x^3+x^2+4x+1$		\\	\hline
	$10$	 		&  	$x^3+3x^2+5x+1$		  \\	\hline
	$22$	 		&  	$3x^4+3x^3+8x^2+7x+1$	  \\	\hline
	$32$ 		&  	$4x^4+6x^3+13x^2+8x+1$	  \\	\hline
	$70$ 		&  	$x^6+6x^5+13x^4+16x^3+23x^2+10x+1$	  \\	\hline
	$102$		& 	$x^6+10x^5+19x^4+29x^3+31x^2+11x+1$	  \\	\hline
	$222$	 	&  	$5x^7+13x^6+39x^5+46x^4+59x^3+46x^2+13x+1$	  \\	\hline
	$324$ 		&  	$6x^7+23x^6+58x^5+75x^4+90x^3+57x^2+14x+1$	  \\	\hline
	$704$ 		& 	$x^9+15x^8+39x^7+97x^6+147x^5+158x^4+153x^3+77x^2+16x+1$	  \\	\hline
	$1028$ 		&	$x^9+21x^8+62x^7+155x^6+222x^5+248x^4+210x^3+91x^2+17x+1$		  \\	
\end{tabular}
\vspace{0.2cm}
\caption{The number of MECs of $W_p$ for $1\leq p \leq 14$, and the associated polynomial generating functions $M(W_p;x)$.}
\label{table: caterpillars}
\vspace{-0.4cm}
\end{table}

\subsection{Complete Binary Trees} 
\label{subsec: complete binary trees}

In the following, we let $T_k$ denote the complete binary tree containing $2^k-1$ nodes and $A_k$ denote the additive tree constructed by adding one leaf to the root node of $T_k$.  
These two trees are depicted in Figure~\ref{fig: complete and additive trees} for $k=3$.

	\begin{figure}[b!]
	\centering
	\includegraphics[width=0.65\textwidth]{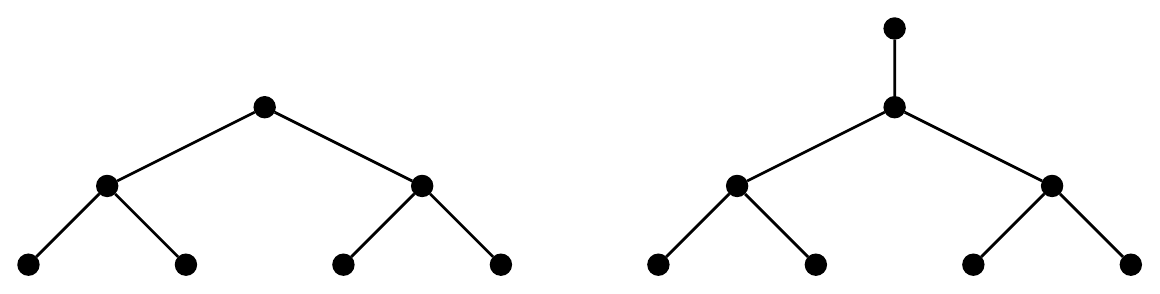}
	\caption{The complete binary tree $T_3$ is depicted on the left and the additive tree $A_3$ is depicted on the right.}
	\label{fig: complete and additive trees}
	\end{figure}

We will now use a series of recursions to enumerate the number of MECs on $T_k$ and $A_k$.  We will then show that the ratio $\frac{M(A_k)}{M(T_k)} < 4$, which means that adding an edge to the root of a complete binary tree increases the number of MECs by at most a factor of 4.  
In practice, we observed that the factor is around 2 for large $k$.  

Before providing a recursion for $M(T_k)$ and $M(A_k)$, we introduce three new graph structures $X_k$, $Y_k$, and $Z_k$ in order to help simplify our recursions.  
Similar to Section~\ref{subsec: stars and bistars}, in the following it will be helpful to label edges that have specified roles in certain MECs.  
The green edges (also labeled with $\square$) indicate that these edges cannot be involved in any immorality.  
The red arrows (also labeled with $\ast$) indicate a fixed immorality in the partially directed graph, and the blue arrows (also labeled with $\circ$) represent fixed arrows that are not in immoralities.  

\begin{enumerate}
\item Let $X_k$ denote the partially directed tree whose skeleton is $A_k$ and for which there is exactly one immorality at the child of the root (note that the root of $A_k$ has degree $1$).    
\item Let $Y_k$ denote the number of MECs on a complete binary tree with $2^k -1 $ nodes such that the root's edges are not involved in any immoralities.
\item Let $Z_k$ denote the number of MECs on an additive tree with $2^k$ nodes such that there are edges directed from the root $r$ to its child $c$ and from $c$ to each of its children.     
\end{enumerate}

	\begin{figure}[t!]
	\centering
	\includegraphics[width=1\textwidth]{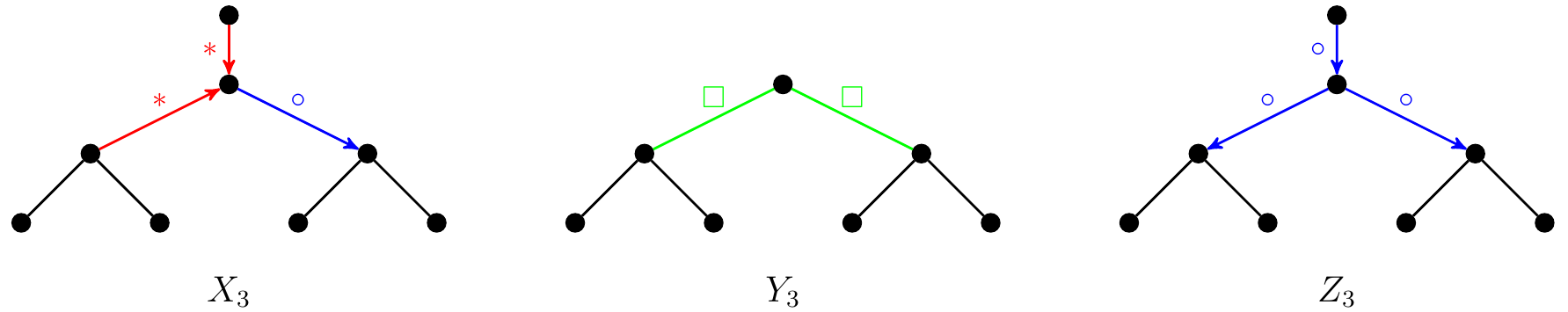}
	\caption{From left-to-right, the graphs $X_3, Y_3$, and $Z_3$.}
	\label{fig: complete and additive subgraphs}
	\end{figure}
	
The graphs $X_3, Y_3$, and $Z_3$ are depicted from left-to-right in Figure~\ref{fig: complete and additive subgraphs}.  
Now we have the following series of recursions for the graphs listed above.

\begin{theorem} 
\label{thm: complete binary trees}      
The following recursions hold for the partially directed graphs $T_k$, $A_k$, $X_k$, $Y_k$, and $Z_k$:
\begin{enumerate}
\item[(a)] $M(T_k) = M(A_{k-1})^2+ M(Y_k)$ with $M(T_1) = 1$,
\item[(b)] $M(A_k) = M(T_k) + 2M(X_k) + M(T_{k-1})^2$ with $M(A_1) = 1$,
\item[(c)] $M(X_k) = M(T_{k-1}) \sqrt{M(Z_k)}$ with $M(X_1) = 1$,
\item[(d)] $M(Y_k) = 2M(Z_{k-1})M(T_{k-1}) - M(Z_{k-1})^2$ with $Y_1 = 1$, and 
\item[(e)] $M(Z_k) = (2M(X_{k-1}) + M(T_{k-2})^2 + M(Z_{k-1}))^2$ with $Z_1 = Z_2 = 1$.
\end{enumerate}
\end{theorem}

We first prove statements $(e), (c), (d)$ in this order and then use them to prove statements $(b)$ and $(a)$.  
\newline

\noindent
\textit{Proof of statement (e)}.  We prove this by analyzing the cases on the left subgraph of $Z_k$ and consider possible immoralities at node $s$ in Figure~\ref{fig: cases for statement 5}. 
	\begin{figure}
	\centering
	\includegraphics[width=1\textwidth]{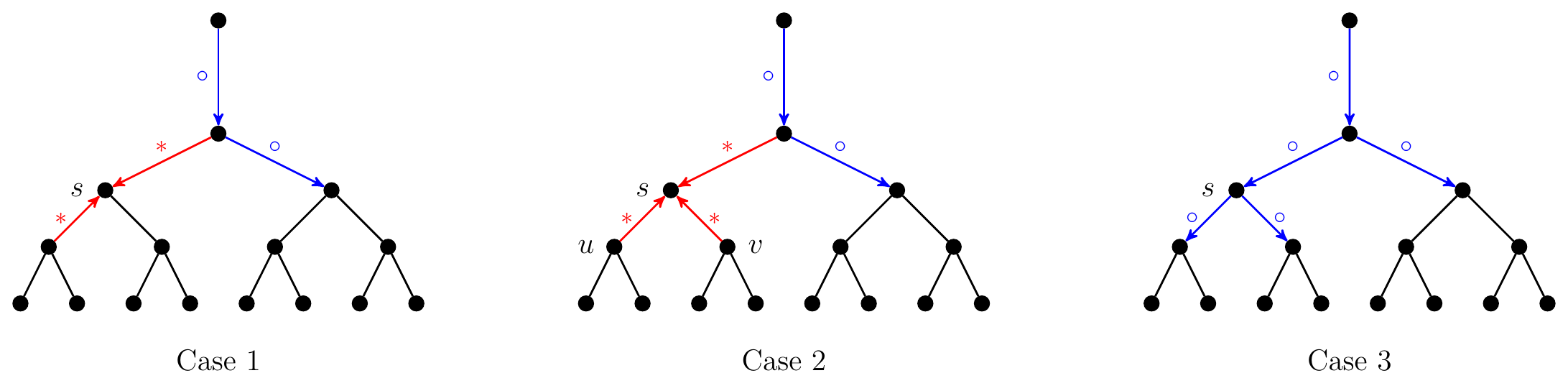}
	\caption{}
	\label{fig: cases for statement 5}
	\end{figure}

\begin{enumerate}
\item If node $s$ has exactly one immorality (as in the leftmost figure), then this substructure contributes exactly $M(X_{k-1})$ MECs.  
By symmetry, there are two ways in which node $s$ can have exactly one immorality, which means these cases contribute $2M(X_{k-1})$ MECs.  

\item If node $s$ has three immoralities (as in the center figure), then this substructure contributes exactly $M(T_{k-2})^2$ MECs as we may treat nodes $u, v$ as roots of complete binary trees $T_{k-2}$. 

\item If node $s$ has no immoralities (as in the rightmost figure), then this substructure contributes exactly $M(Z_{k-1})$ MECs as we may treat the left subgraph as the graph $Z_{k-1}$.  

\end{enumerate}

\noindent
Finally, as we have just considered the cases on the left subgraph of $Z_k$ and as the immoralities on the right subgraph of $Z_k$ are independent of the immoralities on the left subgraph, we square the number of MECs on the left subgraph to conclude  that $M(Z_k) = (2M(X_{k-1}) + M(T_{k-2})^2 + M(Z_{k-1}))^2$.
\newline

\noindent
\textit{Proof of statement (c)}.  Suppose we label two nodes $p$ and $q$ in $X_k$ as in Figure~\ref{fig: cases for statement 3}.  By treating node $p$ as the root of the complete binary tree, and by treating node $q$ as node $s$ in the proof of statement (e), we directly have that $M(X_{k}) = M(T_{k-1}) \sqrt{M(Z_k)}$.
\newline

	\begin{figure}[b!]
	\centering
	\includegraphics[width=0.3\textwidth]{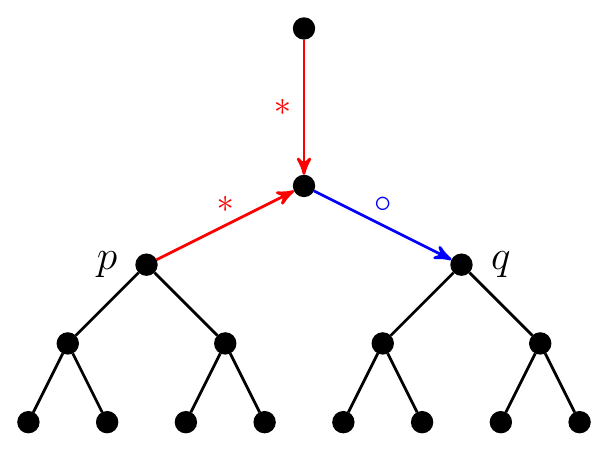}
	\caption{}
	\label{fig: cases for statement 3}
	\end{figure}

\noindent
\textit{Proof of statement (d)}.  We will prove the desired recursion by considering the equivalence classes for which the edges $e_a$ and $e_b$ in Figure~\ref{fig: cases for statement 4} are directed towards the root or away from the root.  

	\begin{figure}[t!]
	\centering
	\includegraphics[width=1\textwidth]{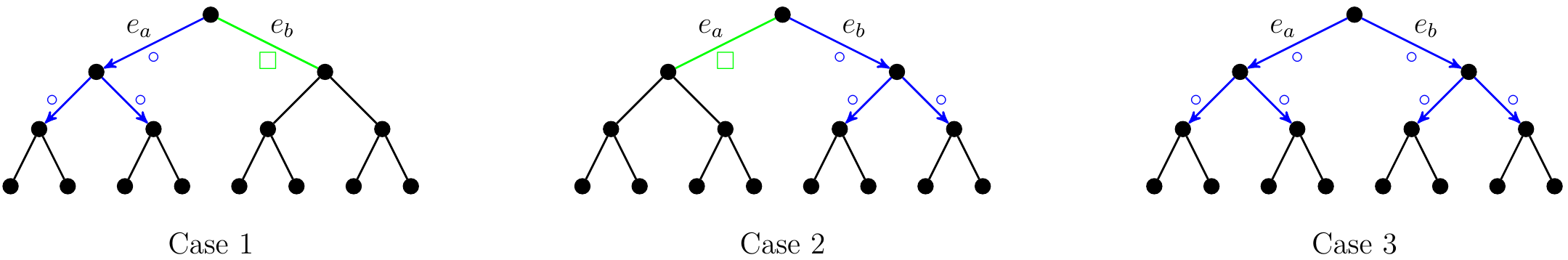}
	\caption{}
	\label{fig: cases for statement 4}
	\end{figure}

\begin{enumerate}
\item Suppose that edge $e_a$ is directed away from the root, then edge $e_b$ can always be directed so that it is not in an immorality at the root's right child.  Thus we can consider the root's right child to be the root of the complete binary tree $T_{k-1}$.  Now since there cannot be an immorality at the root's left child, the left subgraph of the root can be treated as the root of the subgraph $Z_{k-1}$.  This case thus gives us $M(Z_{k-1})M(T_{k-1})$ MECs.
\item Suppose that edge $e_b$ is now directed away from the root, then this case is symmetric to the case above and so there are again $M(Z_{k-1})M(T_{k-1})$ MECs formed.
\item In the above cases we have double-counted the cases where the edges $e_a$ and $e_b$ are both directed away from the root.   Thus we must subtract the number of MECs formed in this case.  However, in this case the left and right subgraphs from the root both represent $Z_{k-1}$.  Thus, there are $M(Z_{k-1})^2$ MECs in this case.  
\end{enumerate}  

\noindent
Hence we have that $M(Y_k) = 2M(Z_{k-1})M(T_{k-1}) - M(Z_{k-1})^2$.
\newline

\noindent
\textit{Proof of statement (b)}.  To prove recursion (b), we will consider the three possible cases of immoralities that can occur at the child $c$ of the root as depicted in Figure~\ref{fig: cases for statement 2}.   

	\begin{figure}[b!]
	\centering
	\includegraphics[width=1\textwidth]{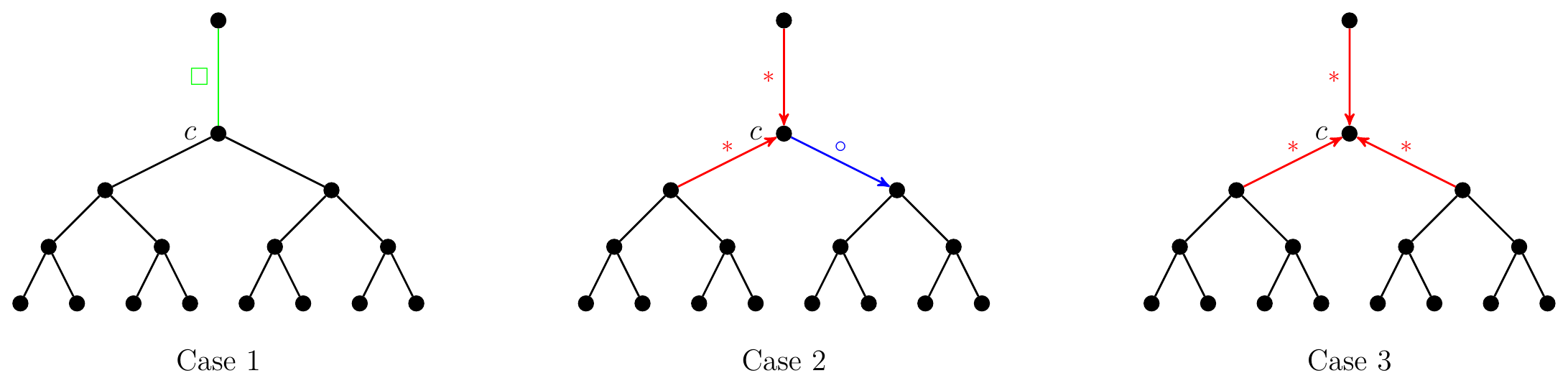}
	\caption{}
	\label{fig: cases for statement 2}
	\end{figure}

\begin{enumerate}
\item In the leftmost figure, if there is no immorality formed by the edge from the root to $c$, then $c$ can be treated as the root of the complete binary tree $T_k$.  This case contributes $M(T_k)$ MECs.
\item In the center figure, if there is exactly one immorality formed by the edge from the root to $s$, then the root can be treated as the root of the tree $X_k$.  This case contributes $2M(X_k)$ MECs, as there are two ways in which the edge from the root to $c$ can be in exactly one immorality.
\item In the rightmost figure, if there are three immoralities formed by the edge from the root to $c$, then the children of $c$ can be treated as roots of complete binary trees $T_{k-1}$.  This case contributes $M(T_{k-1})^2$ MECs. 
\end{enumerate}
Thus, summing over the three cases we have that $M(A_k) = M(T_k) + 2M(X_k) + M(T_{k-1})^2$.
\newline

\noindent
\textit{Proof of statement (a)}.  We can consider the following four cases depicted in Figure~\ref{fig: cases for statement 1} based on the immoralities formed by the root's edges $e_a$ and $e_b$.  

	\begin{figure}[t!]
	\centering
	\includegraphics[width=1\textwidth]{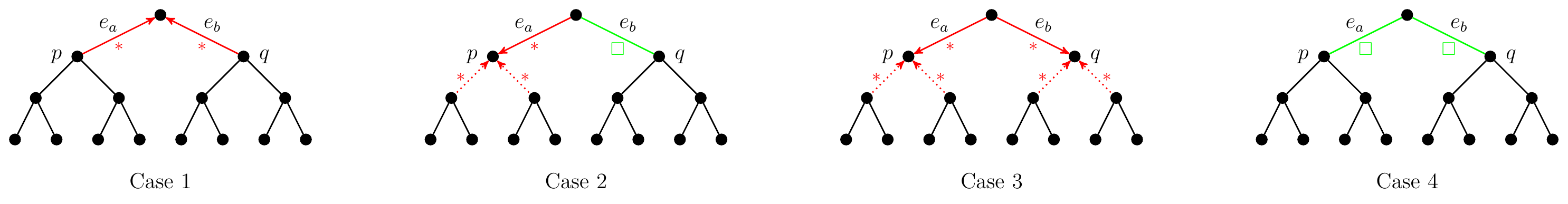}
	\caption{}
	\label{fig: cases for statement 1}
	\end{figure}

\begin{enumerate}
\item If the edges $e_a$ and $e_b$ form an immorality at the root, then the root's children $p$ and $q$ can be treated as roots of complete binary trees $T_{k-1}$.  This case contributes $M(T_{k-1})^2$ MECs.  
\item If the edge $e_a$ forms at least one immorality at $p$ but edge $e_b$ is not in any immoralities, then edge $q$ can be treated as the root of a complete binary tree $T_{k-1}$.  Now $p$ can have exactly one immorality, in which case the left subgraph of the root is the structure $X_{k-1}$ or $p$ can have three immoralities, in which case the children of $p$ can each be treated as the root of a complete binary tree $T_{k-2}$.  Now by symmetry we may consider immoralities formed by the edge $e_b$ as well, which will double the number of MECs formed.  Thus, there are $2M(T_{k-1})[2M(X_{k-1}) + M(T_{k-1})^2]$ MECs.
\item If the edges $e_a$ and $e_b$ form immoralities at $p$ and $q$, then by following the reasoning in the previous case, there are $2M(X_{k-1}) + M(T_{k-2})^2$ MECs formed.  
\item If the edges $e_a$ and $e_b$ form no immoralities, then the remaining graph is simply the structure $Y_k$.  This case contributes $M(Y_k)$ MECs.
\end{enumerate}

\noindent
Summing over the different cases we have that
\begin{equation*} 
\begin{split}
M(T_k) &=  M(T_{k-1})^2  + 2M(T_{k-1})[2M(X_{k-1}) + M(T_{k-1})^2] + 2M(X_{k-1}) \\
		&\hspace{15pt}+ M(T_{k-2})^2 + M(Y_k), \\
&= [M(T_{k-1}) + 2M(X_{k-1}) + M(T_{k-2})^2]^2 + M(Y_k), \\
&= M(A_{k-1})^2 + M(Y_k). \\
\end{split}
\end{equation*}
This completes the proof of Theorem~\ref{thm: complete binary trees}. \hfill$\square$

Now that we have recursions for $T_k$ and $A_k$, we can establish a bound on the number of MECs given by adding an edge to the root of $T_k$ to produce $A_k$.  
In order to do this, we will use the following lemma.

\begin{lemma}
\label{lem: T_k Z_k relation}
For the partially directed graphs $T_k$ and $Z_k$ we have that
$$
M(Z_k) < M(T_k).
$$
\end{lemma}

\begin{proof}
If we omit the root and its edge from the graph $Z_k$, then we see that every MEC formed in $Z_k$ can also be formed in $T_k$.  
Further, since the MEC in $T_k$ with an immorality at the root cannot appear in $Z_k$, we have a strict inequality.  Hence, we have that $M(Z_k) < M(T_k)$.   
\end{proof}

Now we show that adding an edge to the root of $T_k$ increases the number of MECs by at most 4.  

\begin{theorem}
\label{thm: complete/additive tree ratio}
The number of MECs on $A_k$ and $T_k$ satisfy 
$$
1 < \frac{M(A_k)}{M(T_k)} < 4.
$$

\end{theorem}

\begin{proof}
First we let $R_k = \frac{M(A_k)}{M(T_k)}$ and $S_{k-1} = \frac{M(T_k)}{M(T_{k-1})^2}$.
By equation (b) of Theorem~\ref{thm: complete binary trees} we know that
$$M(A_k) = M(T_k) + 2M(X_k) + M(T_{k-1})^2,$$
and hence by equation (c) of Theorem~\ref{thm: complete binary trees}
\begin{equation*} 
\begin{split}
R_k &=  1 + \frac{2M(X_k)}{M(T_k)} + \frac{M(T_{k-1})^2}{M(T_k)},\\
&=  1 + \frac{2M(T_{k-1})\sqrt{M(Z_k)}}{M(T_k)} + \frac{M(T_{k-1})^2}{M(T_k)}.
\end{split}
\end{equation*}
Thus, it follows by Lemma~\ref{lem: T_k Z_k relation} that
$$R_k  < 1 + \frac{2}{\sqrt{S_{k-1}}} + \frac{1}{S_{k-1}}$$
and hence
$$R_k < \left (1 + \frac{1}{\sqrt{S_{k-1}}} \right) ^ 2  < \left (1 + \frac{1}{\sqrt{1}} \right) ^2,$$
which completes the proof.  
\end{proof}

\section{Beyond Trees: Observations for Triangle Free Graphs}
\label{sec: beyond trees}
	\begin{figure}[b!]
	\centering
	\includegraphics[width=.4\textwidth]{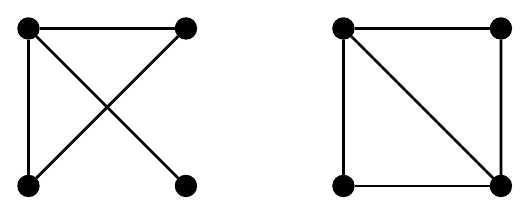}
	\caption{Two graphs with the same polynomials $M(G:x)$.}
	\label{fig: same MEC polynomials}
	\end{figure}
We end this paper with an analysis of the natural generalization of trees, the triangle-free graphs.  
As we will see, much of the intuition for the distribution of immoralities and number of MECs on trees carries over into the more general context of triangle-free graphs.  
However explicitly computing the generating functions $M(G;x)$ and $S(G;x)$ becomes increasingly difficult.  
In Section~\ref{subsec: the complete bipartite graph K_2,p}, we illustrate the increasing level of difficulty in computing these generating functions for triangle-free, non-tree, graphs by computing $M(G;x)$ and $S(G;x)$ for the complete bipartite graph $K_{2,p}$.  
In Section~\ref{subsec: skeletal structure in relation to the number and size of MECs}, we then take a computational approach to this problem, and we study the number and size of MECs relative to properties of the skeleton.  
Using data collected by a program described in \cite{RSU17}, we examine the number and size of MECs on all connected graphs for $p\leq10$ nodes and all triangle-free graphs for $p\leq12$ nodes.  
We compare the number of MECs and their sizes to skeletal properties including average degree, maximum degree, clustering coefficient, and the ratio of the number of immoralities in the MEC to the number of induced $3$-paths in the skeleton.  
For triangle-free graphs, we see that much of the intuition captured by the results of the previous sections extend into this setting.  
In particular, the number and distribution of high degree nodes in a triangle-free skeleton plays a key role in the number and sizes of MECs.  
Finally, unlike $S(G;x)$, we can see using graphs on few nodes that the polynomial $M(G;x)$ is not a complete graph isomorphism invariant for connected graphs on $p$ nodes.  
For instance, the two graphs on four nodes in Figure~\ref{fig: same MEC polynomials} both have $M(G;x) = 1+2x+x^2$.  
However, using this program, we verify that $M(G;x)$ is a complete graph isomorphism invariant for all triangle-free connected graphs on $p\leq 10$ nodes.  
That is, $M(G;x)$ is distinct for each triangle-free connected graph on $p$ nodes for $p\leq10$. 

\subsection{The bipartite graph $K_{2,p}$: a triangle-free, non-tree example}
\label{subsec: the complete bipartite graph K_2,p}
We now give explicit formulae for the number and sizes of the MECs on the complete bipartite graph $K_{2,p}$.  
For convenience, we consider the vertex set of $K_{2,p}$ to be two distinguished nodes $\{a,b\}$ together with the remaining $p$ nodes, labeled by $[p]$, which are collectively referred to as the \emph{spine} of $K_{2,p}$.  
This labeling of $K_{2,p}$ is depicted on the left in Figure~\ref{fig: k-two-p}. 
It is easy to see that the maximum number of immoralities is given by orienting the edges such that all edge heads are at the nodes $a$ and $b$. This results in $m(K_{2,p}) = 2\binom{p}{2}$.
Next, we compute a closed-form formula for the number of MECs for $K_{2,p}$.  
\begin{theorem}
\label{thm: K_2,p number of MECs}
The number of MECs with skeleton $K_{2,p}$ is 
$$
M(K_{2,p}) = \sum_{k=0}^p{p\choose k}\left(2^{p-k}-1+2^k-k\right) - p2^{p-1}.
$$
\end{theorem}

\begin{proof}
To arrive at the desired formula, we divide the problem into three cases:
\begin{enumerate}[(1)]
	\item[(a)] The number of immoralities at node $b$ is ${p\choose 2}$.
	\item[(b)] The number of immoralities at node $b$ is strictly between $0$ and ${p\choose 2}$.
	\item[(c)] There are no immoralities at node $b$.  
\end{enumerate}
Notice that cases (a) and (b) have a natural interpretation via the indegree at node $b$ of the essential graph of the corresponding MECs.  
If the indegree at $b$ is two or more, all edges adjacent to $b$ are essential, and the number of immoralities at node $b$ is given by its indegree.  
Thus, we can rephrase cases (a) and (b) as follows:
\begin{enumerate}[(1)]
	\item[(a)] The indegree of node $b$ in the essential graph of the MEC is $p$.
	\item[(b)] The indegree of node $b$ in the essential graph of the MEC is $1<k<p$.  
\end{enumerate}
In case (a), the MEC is determined exactly by the MEC on the star with center node $a$ and $p$ edges.  
One can easily check (this was also proven as part of Theorem~\ref{thm: stars}) that this yields $2^p-p$ MECs.  

Case (b) is more subtle.  
First, assume that the indegree at node $b$ is $1<k<p$, and the arrows with head $b$ have the tails $\{1,2,\ldots k\}\subset[p]$.  
Then the remaining arrows adjacent to $b$ are all directed outwards with heads $\{k+1,\ldots,p\}$.  
	\begin{figure}
	\centering
	\includegraphics[width=0.75\textwidth]{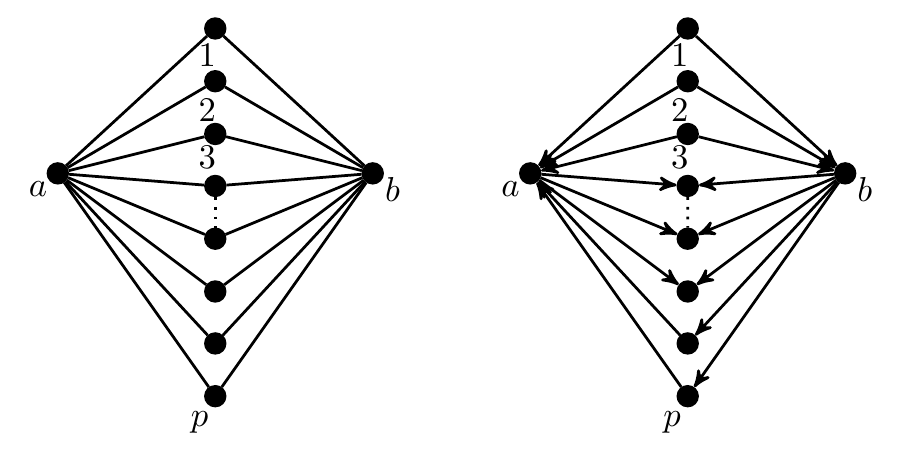}
	\caption{The graph $K_{2,p}$ is depicted on the left, and one of the essential graphs counted in the proof of Theorem~\ref{thm: K_2,p number of MECs} is depicted on the right.}
	\label{fig: k-two-p}
	\end{figure}
Notice that no immoralities can happen at nodes $[k]$ along the spine, but some may occur at the nodes $[p]\backslash[k]$.  
If there are no such immoralities, then node $a$ has indegree $p$, otherwise the essential graph would contain a directed $4$-cycle.  
Similarly, if, without loss of generality, we denote the nodes in $[p]\backslash[k]$ that are the heads of immoralities by $\{k+1,k+2,\ldots,k+s\}$ for $0\leq s<p-k$, then the nodes $k+s+1,\ldots,p$ are tails of the arrows adjacent to node $a$.  
Thus, if the number of immoralities with heads in $[p]\backslash[k]$ is $0\leq s<p-k$, then the immoralities with heads at node $a$ are completely determined.  
Therefore, each $s$-subset of $[p]\backslash[k]$ yields a single MEC.  
Figure~\ref{fig: k-two-p} depicts an example of one such choice of immoralities.  
We start by selecting the arrows to form immoralities at node $b$ which forces the remaining arrows at $b$ to point towards the spine.  
We then select some of these to form immoralities at the spine, and this forces the remaining arrows to be directed inwards towards $a$.  

However, if $s=p-k$, the star induced by nodes $\{a,1,2,\ldots,k\}$ determines the MECs.  
This yields $2^k-k$ classes (see again Theorem~\ref{thm: stars}).  
In total, for case (b) the number of MECs is
$$
\sum_{k=2}^{p-1}{p\choose k}\left(2^{p-k}-1+2^k-k\right).
$$

In case (c), we consider the case when there are no immoralities at node $b$, and we count via placement of immoralities along the spine.  
There are $2^p$ ways to place immoralities along the spine, one for each subset of $[p]$.  
Suppose the immoralities along the spine have the heads $\{1,2,\ldots, k\}$ for $k<p-1$ (the cases $k=p-1$ and $k=p$ are considered separately).  
Then the remaining immoralities can happen at node $a$.  
However, if there is an immorality with head at node $a$ then all other arrows adjacent to $a$ are essential, some of which may point towards the spine with heads in the set $[p]\backslash[k]$.  
Since there are no immoralities with head in the set $[p]\backslash[k]$, then any such outward pointing arrow is part of a directed path from $a$ to $b$.  
However, since there are no immoralities at node $b$, there can be at most one such directed path.  
The presence of any such directed path forces a directed $4$-cycle since $k<p-1$.  
Therefore, for $k<p-1$ the nodes $\{k+1,\ldots,p\}$ must be tails of arrows oriented towards node $a$, thereby yielding only a single MEC.  
Since $k = p$ and $k = p-1$ also yield only a single MEC, case (c) yields a total of $2^p$ classes.  
Combing the total number of MECs counted for each of these cases yields the desired formula.  
\end{proof}

Using the case-by-case analysis from the proof of Theorem~\ref{thm: K_2,p number of MECs} we can count the number of MECs with skeleton $K_{2,p}$ of each possible size.  
Similarly, one can also recover the statistics $m_k(K_{2,p})$ from this proof.  
However, to avoid overwhelming the reader with formulae, we omit the expressions for $m_k(K_{2,p})$.  
\begin{corollary}
\label{cor: K_2,p sizes of MECs}
The possible sizes of a MEC with skeleton $K_{2,p}$ and the number of classes having each size is as follows:
\begin{center}
\begin{tabular}{c | c}
Class size			&	Number of Classes						\\\hline
$1$				&	$2+\sum_{k=2}^{p-1}{p\choose k}2^{p-k}$		\\\hline
$2$				&	$2+{p\choose2}	$						\\\hline
$3\leq k\leq p-1$	&	$1+{p\choose2}$						\\\hline
$p$				&	$2$									\\
\end{tabular}
\end{center}
\end{corollary}

\begin{proof}
Recall the case analysis from the proof of Theorem~\ref{thm: K_2,p number of MECs}.  
In case (a) all MECs are size $1$ except for one which is size $p$.  
This yields $2^p-p-1$ classes of size one and one class of size $p$.  
In case (b), all MECs have size $1$, unless $s=p-k$ and there are no immoralities at node $a$, in which case the class size is $k$.  
This yields ${p\choose k}$ classes of size $k$ for $1<k<p$, and 
$$
\sum_{k=2}^{p-1}{p\choose 2}\left(2^{p-k}-1\right)
$$ 
classes of size $1$.  
In case (c), all MECs have size $p-k$ for $0\leq k<p-1$.  
When $k=p-1$, we get a single class of size $2$, and when $k=p$ we get one more class of size $1$.  
The total number of MECs of size $1$ is then 
\begin{equation*}
\begin{split}
(2^p-p-1)+1+\sum_{k=2}^{p-1}{p\choose k}\left(2^{p-k}-1\right)
	&= (2^p-p-1)+1+\sum_{k=2}^{p-1}{p\choose k}2^{p-k}-\sum_{k=2}^{p-1}{p\choose k},\\
	&= 2^p+2+\sum_{k=2}^{p-1}{p\choose k}2^{p-k}-\sum_{k=0}^{p}{p\choose k},\\
	&= 2+\sum_{k=2}^{p-1}{p\choose k}2^{p-k}.\\
\end{split}
\end{equation*}
The other formulae are quickly realized from the above arguments.  
\end{proof}

\subsection{Skeletal structure in relation to the number and size of MECs}
\label{subsec: skeletal structure in relation to the number and size of MECs}

We now take a computational approach to analyzing the number and size of MECs on triangle-free graphs with respect to their skeletal structure.  
The data analyzed here was collected using the program described in \cite{RSU17}, and this program can be found at \url{https://github.com/aradha/mec_generation_tool}. 
The results of \cite{RSU17}, and those provided in the previous sections of this paper, indicate that the number and distribution of high degree nodes in a triangle-free graph dictate the size and number of MECs allowable on the skeleton.  
%
In this section, we parse these observations in terms of the data collected via our computer program.  

	\begin{figure}
	\centering
	$\begin{array}{c c c}
	\includegraphics[width=0.45\textwidth]{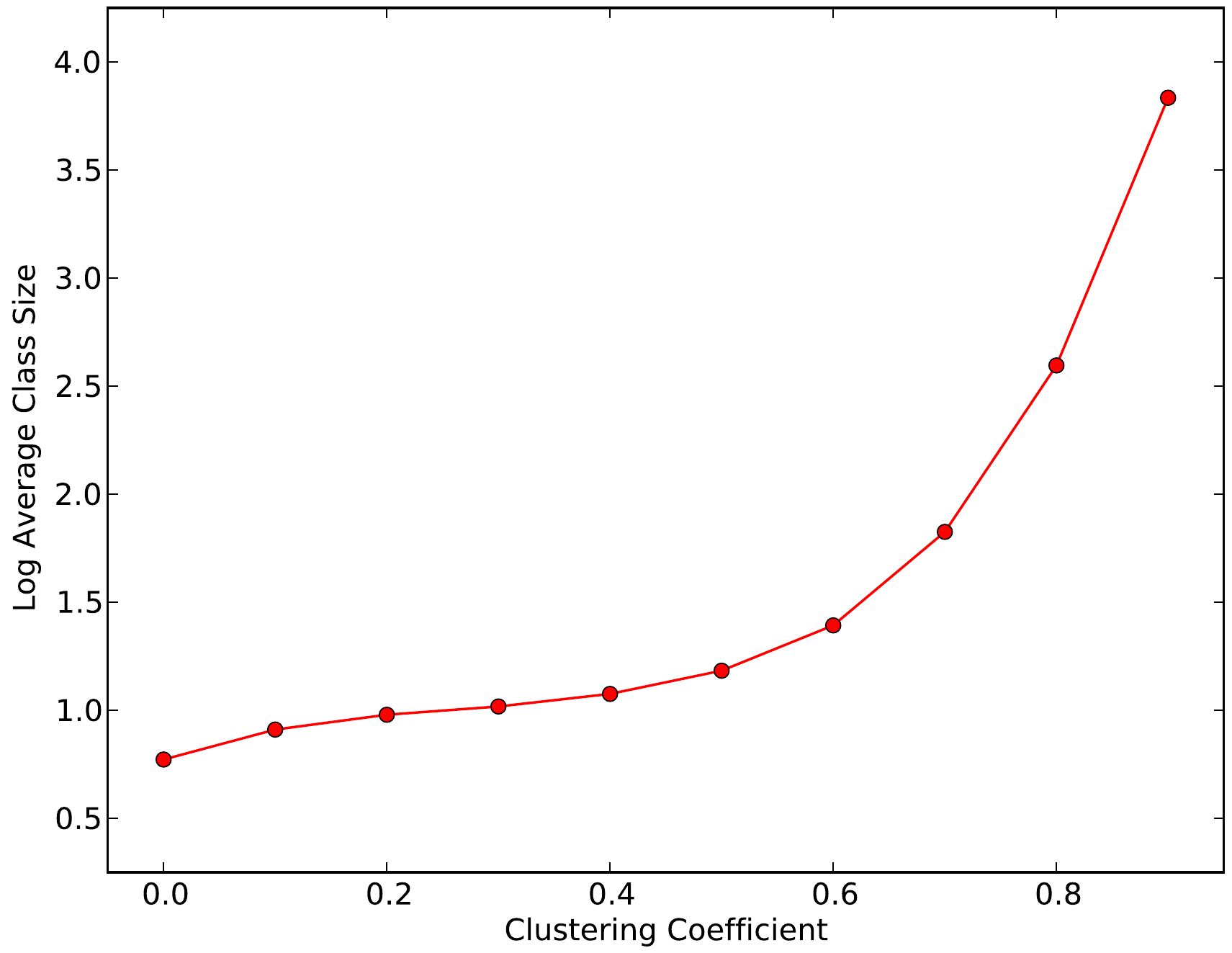} &
	\includegraphics[width=0.47\textwidth]{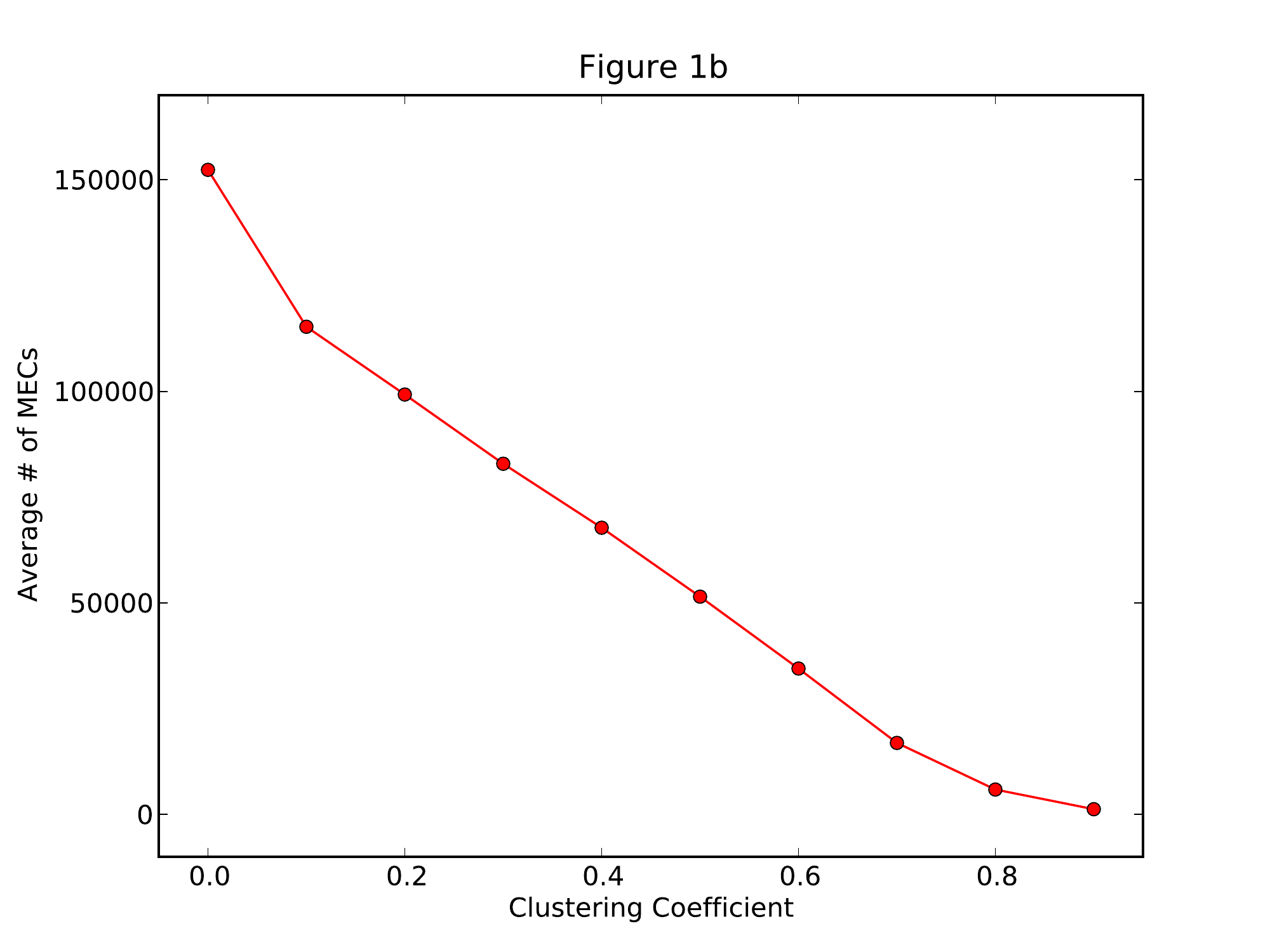} \\
	\end{array}$
	\caption{Clustering coefficient as compared to log average class size and the average number of MECs for connected graphs with $p\leq 10$ nodes and 25 edges.}
	\label{fig: clustering coefficient all graphs}
	\end{figure}
Recall that the \emph{(global) clustering coefficient} of a graph $G$ is defined as the ratio of the number of triangles in $G$ to the number of connected triples of vertices in $G$.  
The clustering coefficient serves as a measure of how much the nodes in $G$ cluster together.  
Figure~\ref{fig: clustering coefficient all graphs} presents two plots: one compares the clustering coefficient to the log average class size and the other compares it to the average number of MECs.  
This data is taken over all connected graphs on $p\leq 10$ nodes with $25$ edges (to achieve a large number of MECs).  
As we can see, the average class size grows as the clustering coefficient increases.  
This is to be expected, since an increase in the number of triangles within the DAG should correspond to an increase in the size of the chain components of the essential graph.  
On the other hand, the average number of MECs decreases with respect to the clustering coefficient, which is to be expected given that the class sizes are increasing.    
This decrease in the average number of MECs empirically captures the intuition that having many triangles in a graph results in fewer induced $3$-paths, which represent the possible choices for distinct MECs with the same skeleton.  

	\begin{figure}[!t]
	\centering
	$
	\begin{array}{c c c}
	\includegraphics[width=0.45\textwidth]{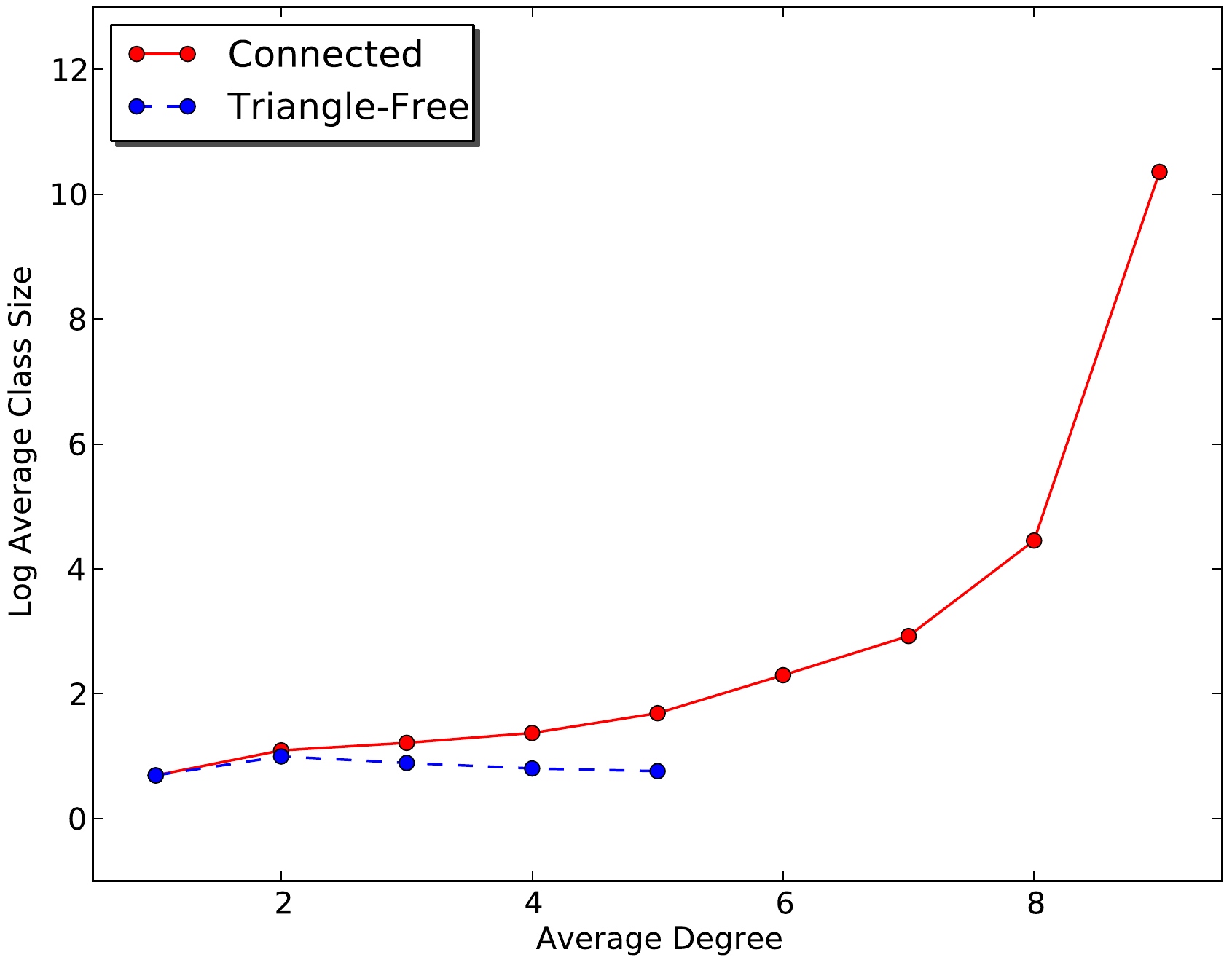} &
	\includegraphics[width=0.48\textwidth]{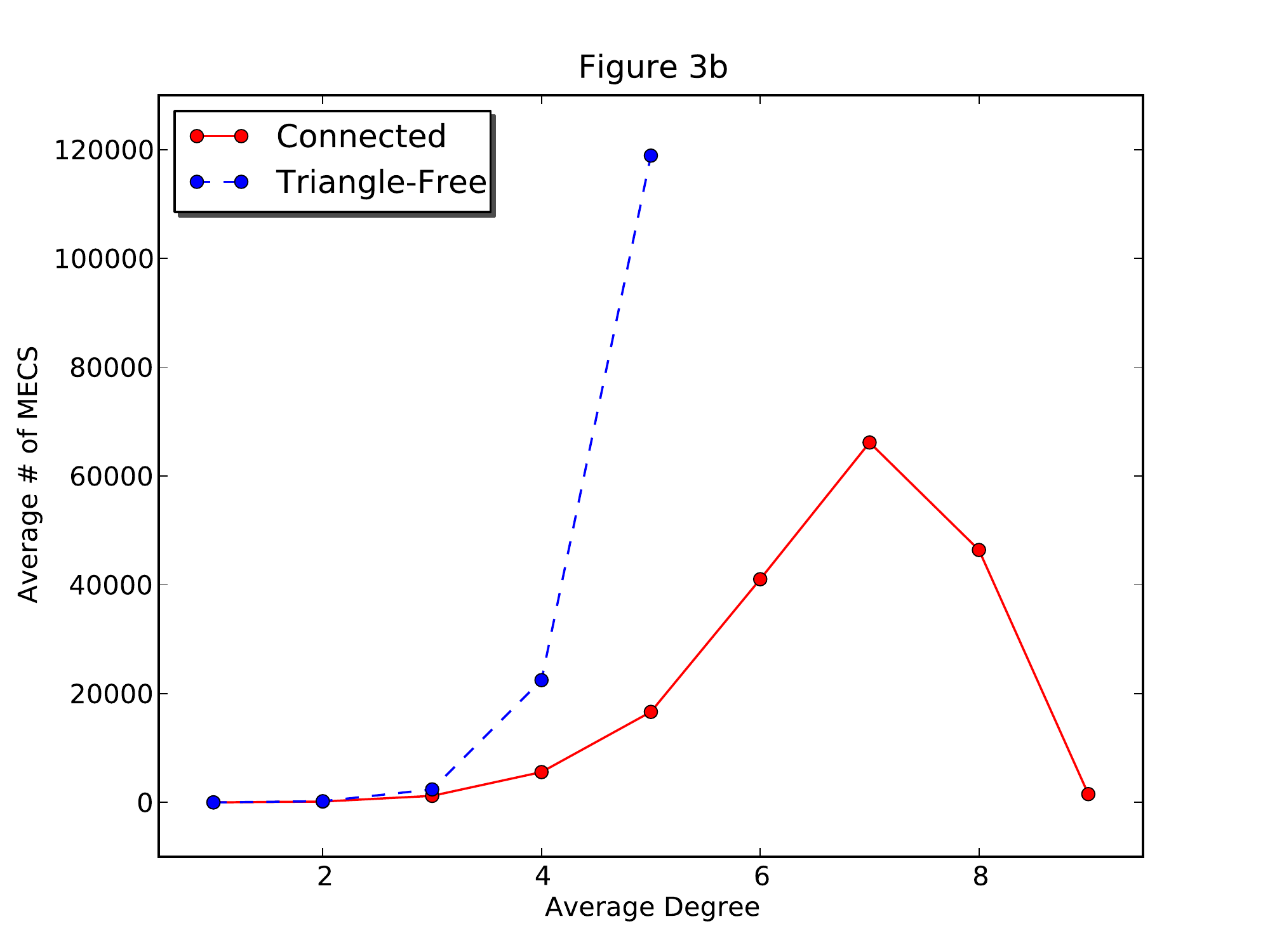} \\
	\end{array}
	$
	\caption{Average degree versus log average class size and average number of MECs for all graphs and triangle-free graphs on $10$ nodes.}
	\label{fig: average degree vs log average class size}
	\end{figure}
Figure~\ref{fig: average degree vs log average class size}  presents a pair of plots, the first of which  compares the average degree of the underlying skeleton of the DAG to the log average class size of the associated MEC.  
The second plot compares the average degree of the skeleton to the average number of MECs it supports.  
Both plots present one curve for all connected graphs and a second curve for triangle-free graphs on $10$ nodes.  
For connected graphs on $10$ nodes the left-most plot shows a strict increase in the log average MEC class size as the average degree of the nodes in the underlying skeleton increases.  
This is to be expected since graphs with a higher average degree are more likely to contain larger chain components.  
On the other hand, the average class size for triangle-free graphs increases for average degree up until approximately $2.0$, and then shows a steady decrease for larger average degree.  
Since the average degree of a tree on $p$ nodes is $2-\frac{2}{p}$, this suggests that the largest MECs amongst triangle-free graphs have skeleta being trees.  
As such, the bounds developed in Section~\ref{sec: bounding the size and number of mecs on trees} of this paper can be, heuristically, thought to apply more generally to all triangle-free graphs.  

The right-most plot in Figure~\ref{fig: average degree vs log average class size} describes the relationship between average degree and the average number of MECs for all connected graphs and triangle-free graphs on $10$ nodes.  
We see from this that in the setting of all connected graphs, the skeleta with the largest average number of MECs appear to have average degree $7$, whereas in the triangle-free setting, the higher the average degree the more equivalence classes the skeleta can support.  
This supports the intuition that the more high degree nodes there are in a triangle-free graph, the more equivalence classes the graph can support.  
	\begin{figure}[t!]
	\centering
	$\begin{array}{c c c}
	\includegraphics[width=0.45\textwidth]{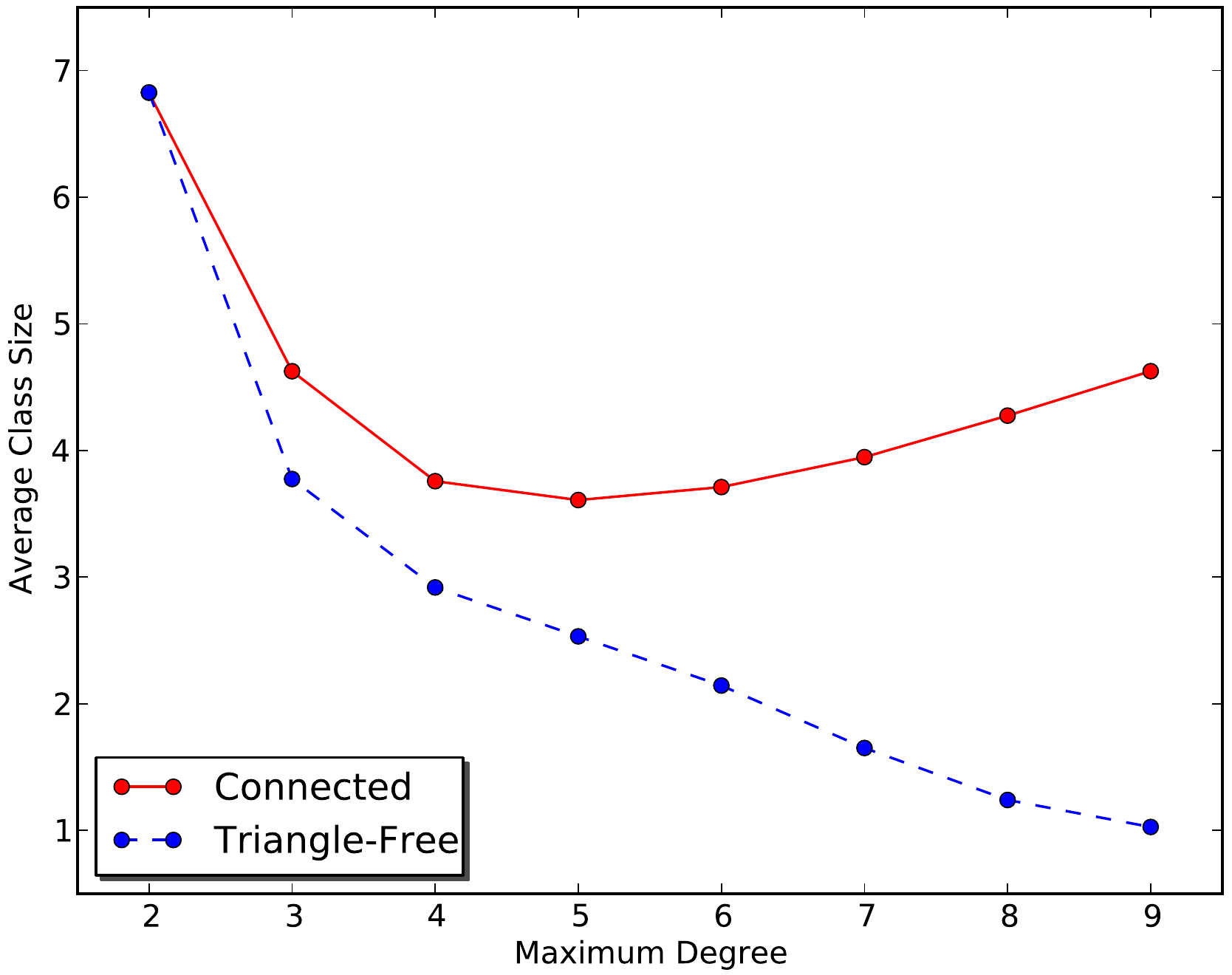} &
	\includegraphics[width=0.48\textwidth]{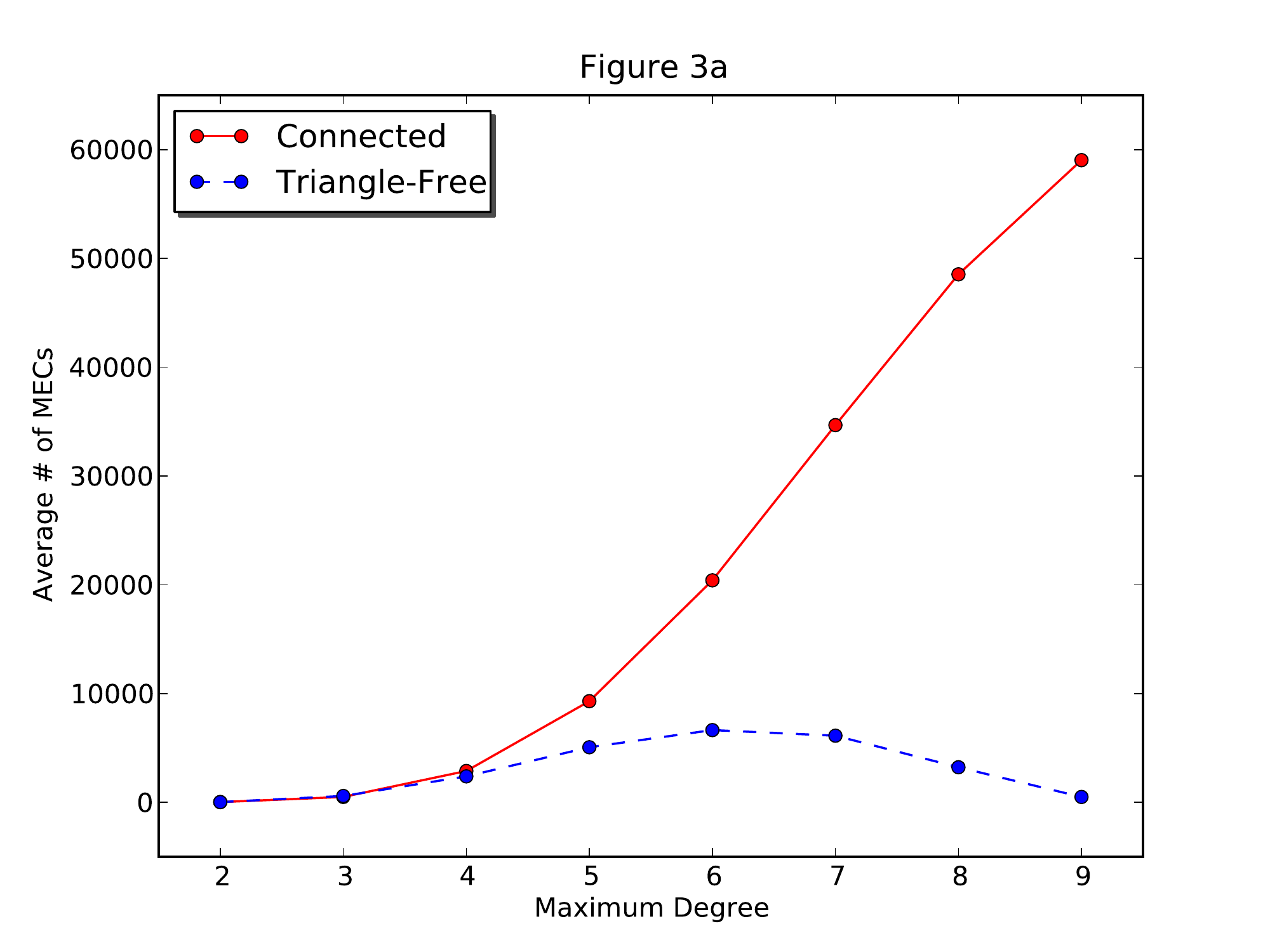} \\
	\end{array}$
	\caption{Maximum degree versus log average class size and average number of MECs for all graphs and triangle-free graphs on $10$ nodes.}
	\label{fig: max degree vs log average class size}
	\end{figure}

The left-most plot in Figure~\ref{fig: max degree vs log average class size} depicts the relationship between the maximum degree of a node in a skeleton and the average class size on the skeleton for all connected graphs and for triangle-free graphs on $10$ nodes.  
For all graphs, the relationship appears to be almost linear beginning with maximum degree $5$, suggesting that average class size grows linearly with the maximum degree of the underlying skeleton.  
This growth in class size is due to the introduction of many triangles as the maximum degree grows.  
On the other hand, in the triangle-free setting we actually see a decrease in average class size as the maximum degree grows, which empirically reinforces this intuition.  

The right-most plot in Figure~\ref{fig: max degree vs log average class size} records the relationship between the maximum degree of a node in a skeleton and the average number of MECs supported by that skeleton for all connected graphs and triangle-free graphs on at most $10$ nodes.  
For all graphs, we see that the average number of MECs grows with the maximum degree of the graphs, and this growth is approximately exponential.  
In the triangle-free setting, the average number of MECs appears to be unimodal, but would be increasing if we considered also all graphs on $p>10$. For triangle-free graphs there is only one graph with maximum degree 9, namely the star $G_1(9)$, where the number of MECs is $2^9-9$. For connected graphs the average number of MECs is pushed up by those cases consisting of a complete bipartite graph where in addition one node is connected to all other nodes. 

	\begin{figure}
	\centering
	\includegraphics[width=0.46\textwidth]{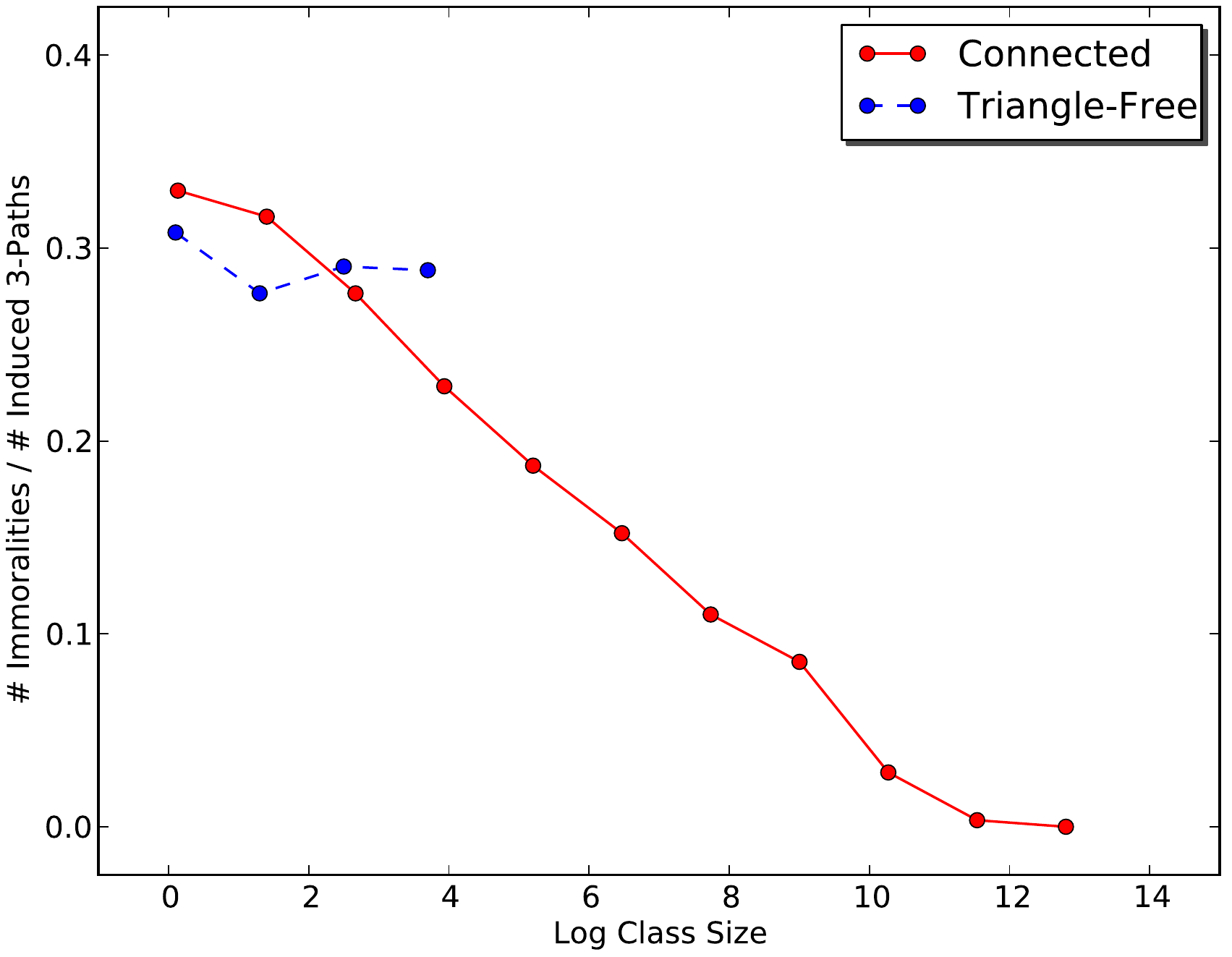}
	\caption{Class size versus the ratio of the number of immoralities to the number of induced $3$-paths in all MECs on 10 nodes.}
	\label{fig: class size versus immorality ratio}
	\end{figure}
The final plot of interest is in Figure~\ref{fig: class size versus immorality ratio}, and it shows the relationship between MEC size and the ratio of the number of immoralities in the MEC to the number of induced $3$-paths in the skeleton for all connected graphs and triangle-free graphs on $10$ nodes.  
That is, it shows the relationship between the class size and how many of the potential immoralities presented by the skeleton are used by the class.  
It is interesting to note that, in the triangle-free setting, as the class size grows, this ratio appears to approach $0.3$, suggesting that most large MECs use about a third of the possible immoralities in triangle-free graphs.   
In the connected graph setting, as the class size grows, we see a steady decrease in the value of this ratio.  
This supports the intuition that a larger class size corresponds to an essential graph with large chain components and few immoralities.

\bigskip
\noindent
{\bf Acknowledgements}. 
We wish to thank Brendan McKay for some helpful advice in the use of the programs {\tt nauty} and {\tt Traces} \cite{MP14}.  
Adityanarayanan Radhakrishnan was supported by ONR (N00014-17-1-2147). Liam Solus was partially supported by an NSF Mathematical Sciences Postdoctoral Research Fellowship (DMS - 1606407). 
Caroline Uhler was partially supported by DARPA (W911NF-16-1-0551), NSF (1651995), and ONR (N00014-17-1-2147).

\end{document}